\documentclass[12pt]{amsart}
\usepackage{amscd}     
\usepackage{amssymb}
\usepackage{amsmath, amsthm, graphics}
\usepackage{xypic}     
\LaTeXdiagrams        
\usepackage[all]{xy}
\xyoption{2cell} \UseAllTwocells \xyoption{frame} \CompileMatrices
\allowdisplaybreaks[3]
\usepackage{amsfonts}
\usepackage[normalem]{ulem}
\usepackage{hyperref}
\usepackage{tikz}
\usepackage{mathtools}
\usepackage{verbatim}

\usepackage{latexsym}
\usepackage{epsfig}
\usepackage{amsfonts}
\usepackage{enumerate}
\usepackage{times}

\newtheorem{prop}{Proposition}
\newtheorem{lem}[prop]{Lemma}

\newtheorem{cor}[prop]{Corollary}
\newtheorem{thm}[prop]{Theorem}

\newtheorem{defn}[prop]{Definition}

\theoremstyle{remark}

\theoremstyle{remark}

\numberwithin{equation}{section}

\newcommand{\com}{\mathbb{C}}

\newcommand{\T}{\mathsf{T}}

\newcommand{\Hilb}{\mathsf{Hilb}^n(\mathbb{C}^2)}
\newcommand{\Sym}{\mathsf{Sym}^n}

\newcommand{\hilbnc}{\mathsf{Hilb}^n(\mathbb{C}^2)}

\newcommand{\cF}{\mathcal{F}}

\def\<{\left\langle}
\def\>{\right\rangle}

\newcommand{\CC}{\mathbb{C}}

\def\ch{{\rm ch}}

\def\b1{{\mathbf 1}}

\begin{document}

\title[The $\mathsf{Hilb}/\mathsf{Sym}$
correspondence for $\CC^2$: descendents and
Fourier-Mukai]{The $\mathsf{Hilb}/\mathsf{Sym}$
correspondence for $\CC^2$: \\descendents and
Fourier-Mukai}

\author[Pandharipande]{Rahul Pandharipande}
\address{Department of  Mathematics\\ ETH Z\"urich \\ R\"amistrasse 101 \\ 8092 Z\"urich\\ Switzerland}
\email{rahul@math.ethz.ch}

\author[Tseng]{Hsian-Hua Tseng}
\address{Department of Mathematics\\ Ohio State University\\ 100 Math Tower, 231 West 18th Ave. \\ Columbus,  OH 43210\\ USA}
\email{hhtseng@math.ohio-state.edu}

\date{August 2019}

\begin{abstract}
    We study here the crepant resolution correspondence for the $\T$-equivariant descendent Gromov-Witten theories of $\Hilb$ and $\Sym(\CC^2)$.
The descendent correspondence is obtained from
our previous matching of the  associated CohFTs
by applying  Givental's 
quantization formula to a specific
symplectic
transformation
$\mathsf{K}$. The first result of the
paper is an explicit computation of $\mathsf{K}$.
Our main result then
establishes a fundamental relationship between the
Fourier-Mukai equivalence of the associated derived categories
(by Bridgeland, King, and Reid) and the
symplectic transformation $\mathsf{K}$
via Iritani's integral structure. The results
use Haiman's Fourier-Mukai calculations and
are exactly aligned with Iritani's point of view on
crepant resolution.
\end{abstract}

\maketitle
\setcounter{tocdepth}{1}
\tableofcontents

\setcounter{section}{-1}
\section{Introduction}

\subsection{Overview}
The diagonal action on $\CC^2$ of the torus $\T =(\CC^*)^2$ lifts canonically to the Hilbert scheme of $n$ points $\Hilb$ and the
orbifold symmetric product
$$\Sym(\CC^2)=[(\CC^2)^n/\Sigma_n]\, .$$
Both the Hilbert-Chow morphism  
\begin{equation}\label{xxx9}
\Hilb \rightarrow (\CC^2)^n/\Sigma_n
\end{equation}
and the coarsification morphism
\begin{equation}\label{xxx9.5}
\Sym(\CC^2) \rightarrow (\CC^2)^n/\Sigma_n
\end{equation}
are $\T$-equivariant crepant resolutions of the singular
quotient variety $(\CC^2)^n/\Sigma_n$.

The geometries of the two crepant resolutions 
$\Hilb$ and $\Sym(\CC^2)$ are connected in many beautiful ways.
The classical McKay correspondence \cite{MK} provides
an isomorphism on the level of $\T$-equivariant cohomology: $\T$-equivariant singular cohomology for $\Hilb$ and 
$\T$-equivariant Chen-Ruan orbifold
cohomology for $\Sym(\CC^2)$.
A
lift of the McKay correspondence to an equivalence of 
$\T$-equivariant derived categories was proven by
Bridgeland, King, and Reid \cite{bkr}
using a Fourier-Mukai transformation.

Quantum cohomology provides a different enrichment of the McKay correspondence. For the crepant resolutions
$\Hilb$ and $\Sym(\CC^2)$, the genus 0
equivalence of the $\T$-equivariant Gromov-Witten
theories was proven in \cite{bg} using \cite{bp,op}.
Going further, the crepant resolution correspondence in all 
genera
was proven in \cite{pt} by matching the associated
$\mathsf{R}$-matrices and
Cohomological Field Theories (CohFTs), see \cite[Section 4]{p-icm} for
a survey.

The results of \cite{bg,pt} concern the $\T$-equivariant
Gromov-Witten theory with {\em primary} insertions.
However, following a remarkable proposal of Iritani,
to see the connection between the Fourier-Mukai
transformation of \cite{bkr} and the 
crepant resolution correspondence for Gromov-Witten
theory,  {\em descendent} insertions are required.
Our first result here is a determination
of the crepant resolution correspondence for the
$\T$-equivariant Gromov-Witten theories of $\Hilb$ and $\Sym(\CC^2)$
with {descendent} insertions via a symplectic
transformation
$\mathsf{K}$ which we compute explicitly.
The main result of the paper is a proof of 
a fundamental relationship between the
Fourier-Mukai equivalence of the associated derived categories
\cite{bkr} and the
symplectic transformation $\mathsf{K}$
via Iritani's integral structure. The results
use Haiman's Fourier-Mukai calculations \cite{haiman_cdm,haiman_sym2001} and
are exactly aligned with Iritani's point of view on
crepant resolutions \cite{i,i2}.

\subsection{Descendent correspondence}
The descendent correspondence 
for the $\T$-equivariant Gromov-Witten theories
of  $\Hilb$ and $\Sym(\CC^2)$
is obtained from the CohFT matching of \cite{pt} together with the quantization formula of Givental \cite{g}. Our first result is a formula for the symplectic transformation
$$\mathsf{K} \in \text{Id} + z^{-1}\cdot \text{End}(H^*_\T(\Hilb))[[z^{-1}]]$$
defining the descendent correspondence.{\footnote{Cohomology
will always be taken here with $\CC$-coefficients.}} 

The formula for $\mathsf{K}$ is best described in terms of the Fock space $\cF$ which is freely generated over $\mathbb{C}$ by commuting creation operators $\alpha_{-k}$ for $k\in \mathbb{Z}_{>0}$ acting on the vacuum vector $v_{\emptyset}$. The annihilation operators $\alpha_k, k\in \mathbb{Z}_{>0}$ satisfy $$\alpha_k\cdot v_\emptyset=0\,,\ \ \  k>0$$ and commutation relations $$[\alpha_k,\alpha_l]=k\delta_{k+l}\, .$$
The Fock space $\cF$ admits an additive basis $$|\mu\rangle=\frac{1}{\mathfrak{z}(\mu)}\prod_i \alpha_{-\mu_i}v_\emptyset\, , \quad \mathfrak{z}(\mu)=|\text{Aut}(\mu)| \prod_i \mu_i\, ,$$ indexed by partitions $\mu=(\mu_1, \mu_2,...)$.

An additive isomorphism 
\begin{equation}\label{ffkk}
\cF \otimes _{\mathbb C} {\mathbb C}[t_1,t_2]\stackrel{\sim}{=} 
\bigoplus_{n\geq 0} H_{\T}^*(\Hilb)\, ,
\end{equation}
is given by identifying $|\mu\rangle$ on the left
with the corresponding Nakajima basis elements on the right. The intersection pairing $(-,-)^\text{Hilb}$ on the $\T$-equivariant cohomology of $\hilbnc$ induces a pairing on Fock space,  $$\eta( \mu,\nu)=\frac{(-1)^{|\mu|-\ell(\mu)}}{(t_1t_2)^{\ell(\mu)}}\frac{\delta_{\mu\nu}}{\mathfrak{z}(\mu)}\, .$$

In the following result, we write the formula for $\mathsf{K}$
in terms of the Fock space,
$$\mathsf{K}\in \text{Id}+z^{-1}\cdot \text{End}(\cF\otimes_{\mathbb C} {\mathbb C}[t_1,t_2])[[z^{-1}]], ,$$ 
using
\eqref{ffkk}.  

\begin{thm}\label{kkkk}
The descendent correspondence is determined by the symplectic transformation $\mathsf{K}$ given by the formula
\begin{equation*}
\mathsf{K}\left(\mathsf{J}^\lambda \right)=\frac{z^{|\lambda|}}{(2\pi\sqrt{-1})^{|\lambda|}}\left(\prod_{\mathsf{w}: \T\text{-weights of } \text{Tan}_\lambda\hilbnc}\Gamma(\mathsf{w}/z+1)\right)\spadesuit\mathsf{H}^\lambda_z\, .
\end{equation*}
\end{thm}
Here, $\mathsf{J}^\lambda$ is the Jack symmetric function
defined by equation (\ref{eqn:jack}) of Section
\ref{qqddee}, and $\mathsf{H}^\lambda_z$ is the Macdonald polynomial\footnote{The footnote $z$ indicates a rescaling of the parameters, 
$\mathsf{H}^\lambda_z = \mathsf{H}^\lambda(\frac{t_1}{z},\frac{t_2}{z})$.}, see \cite{haiman_cdm,mac,op29}. The linear operator $$\spadesuit: \cF\to \cF$$ is defined by 
\begin{equation*}
\spadesuit |\mu\rangle= z^{\ell(\mu)} \frac{(2\pi\sqrt{-1})^{\ell(\mu)}}{\prod_i \mu_i}\prod_i \frac{\mu_i^{\mu_it_1/z}\mu_i^{\mu_i t_2/z}}{\Gamma(\mu_i t_1/z)\Gamma(\mu_i t_2/z)}|\mu\rangle   \, .
\end{equation*}

The descendent correspondence in genus $0$, expressed in terms of Givental's Lagrangian cones, is explained\footnote{See for (\ref{eqn:F_isom}) the definition of the symplectic isomorphism $\mathsf{C}$.} 
in Theorem \ref{thm:genus0_corresp} of
Section \ref{sec:genus0},
\begin{equation*}
\mathcal{L}^\text{Sym}=\mathsf{C}\mathsf{K}q^{-D/z}\mathcal{L}^\text{Hilb}\, ,
\end{equation*}
where $D=-|(2, 1^{n-2})\rangle$ is the $\T$-equivariant first Chern class of the tautological vector bundle on $\hilbnc$. 
The descendent correspondence for all $g$, formulated in terms of generating series, 
\begin{equation*}
e^{-F_1^\text{Sym}(\tilde{t})}\mathcal{D}^\text{Sym}=\widehat{\mathsf{C}}\, \widehat{\mathsf{K}}\, \widehat{q^{-D/z}}\,
\left(e^{-F_1^\text{Hilb}(t_D)}\mathcal{D}^\text{Hilb}\right)\, ,
\end{equation*}
is discussed in Theorem \ref{thm:higher_genus_corresp} of Section \ref{sec:higher_genus}.

For toric crepant resolutions, the symplectic transformation underlying the descendent correspondence is constructed in \cite{cij} by using explicit slices of Givental's Lagrangian cones constructed via the Toric Mirror Theorem \cite{ccit,g0}. 
We proceed differently here. The symplectic transformation $\mathsf{K}$ is constructed by comparing the two fundamental solutions $\mathsf{S}^\text{Hilb}$ and $\mathsf{S}^\text{Sym}$ of the
QDE given by descendent Gromov-Witten invariants of $\hilbnc$ and $\Sym(\CC^2)$ respectively. Via
the $\mathsf{Hilb}/\mathsf{Sym}$ correspondence
in genus 0, Theorem \ref{kkkk} is then simply a reformulation
of the calculation of the connection matrix in \cite[Theorem 4]{op29}.


\subsection{Fourier-Mukai}
An equivalence of $\T$-equivariant derived categories 
\begin{equation*}
\mathbb{FM}: D^b_\T(\hilbnc)\to D^b_\T(\text{Sym}^n(\mathbb{C}^2))    
\end{equation*}
is constructed by Bridgeland, King, and Reid in \cite{bkr} via a
tautological Fourier-Mukai kernel.
We also denote by 
$\mathbb{FM}$
the induced isomorphism on $\T$-equivariant $K$-groups,
\begin{equation}\label{xx99}
\mathbb{FM}: K_\T(\hilbnc)\to K_\T(\text{Sym}^n(\mathbb{C}^2))\, .  
\end{equation}

Iritani \cite{i} has proposed a beautiful
framework  for the crepant
resolution correspondence. In the case of
$\Hilb$ and $\Sym(\CC^2)$, the
isomorphism \eqref{xx99} 
on $K$-theory should 
be related to a symplectic transformation $$\mathcal{H}^\text{Hilb}\to \mathcal{H}^\text{Sym}$$ via Iritani's integral structure. 
The Givental spaces $\mathcal{H}^\text{Hilb}$ and $\mathcal{H}^\text{Sym}$ will be defined
below (in a multivalued form).
A discussion of Iritani's perspective can be found in \cite{i2}.
Our main result is a formulation and proof of Iritani's proposal
for the crepant resolutions $\Hilb$ and $\Sym(\CC^2)$.
For the precise statement, further definitions are required.

\vspace{5pt}
\noindent $\bullet\ $ Define the operators $\text{deg}^{\text{Hilb}}_0$, $\rho^\text{Hilb}$, and $\mu^\text{Hilb}$ on $H_\T^*(\hilbnc)$ as follows. For $\phi\in H_\T^k(\hilbnc)$, 
\begin{equation*}
\begin{split}
    &\text{deg}^{\text{Hilb}}_0(\phi)= k\phi\, ,\\
    &\mu^\text{Hilb}(\phi)=\left(\frac{k}{2}-\frac{2n}{2}\right)\phi\, ,\\
    &\rho^\text{Hilb}(\phi)= c_1^\T(\hilbnc)\cup\phi\, .
    \end{split}    
\end{equation*} 
The {\em multi-valued Givental space} 
$\widetilde{\mathcal{H}}^\text{Hilb}$
for $\hilbnc$ is defined by 
$$\widetilde{\mathcal{H}}^\text{Hilb}=H_\T^*(\hilbnc, \mathbb{C})\otimes_{\mathbb{C}[t_1,t_2]}\mathbb{C}(t_1,t_2)[[\log(z)]]((z^{-1}))\, .$$

\begin{defn}
Let $\Psi^\text{\em Hilb}: K_\T(\hilbnc)\to \widetilde{\mathcal{H}}^\text{\em Hilb}$ be defined by 
\begin{equation*}
  \Psi^\text{\em Hilb}(E)= z^{-\mu^\text{\em Hilb}}z^{\rho^\text{\em Hilb}}\left(\Gamma_\text{\em Hilb}\cup (2\pi \sqrt{-1})^{\frac{\text{\em deg}^{\text{\em Hilb}}_0}{2}}\text{\em ch}(E)\right)\, ,  
\end{equation*}
where $\text{\em ch}(-)$ is the $\T$-equivariant Chern character, $\Gamma_\text{Hilb}\in H_\T^*(\hilbnc)$ is the $\T$-equivariant Gamma class of $\hilbnc$ of \cite[Section 3.1]{cij}, and 
the operators 
$$z^{-\mu^\text{\em Hilb}}:\widetilde{\mathcal{H}}^\text{\em Hilb}\to \widetilde{\mathcal{H}}^\text{\em Hilb}\,,
\ \ \ \ \
z^{\rho^\text{\em Hilb}}:\widetilde{\mathcal{H}}^\text{\em Hilb}\to \widetilde{\mathcal{H}}^\text{\em Hilb}$$ are defined by 
$$z^{-\mu^\text{\em Hilb}}=\sum_{k\geq 0}\frac{\left(-\mu^\text{\em Hilb}\log z\right)^k}{k!}\, ,
\ \ \ \ \
z^{\rho^\text{Hilb}}=\sum_{k\geq 0}\frac{\left(\rho^\text{Hilb}\log z\right)^k}{k!}\, .$$
\end{defn}
Since $|\mu\rangle$ is identified with the 
corresponding Nakajima basis element, we have $$\text{deg}^{\text{Hilb}}_0|\mu\rangle=2(n-\ell(\mu))|\mu\rangle\, .$$
Also, since $t_1, t_2$ both have degree $2$, we have $$\text{deg}^{\text{Hilb}}_0t_1=2=\text{deg}^{\text{Hilb}}_0t_2\, .$$

\vspace{5pt}
\noindent $\bullet\ $ Define the operators\footnote{In the definition of $\rho^\text{Sym}$ we denote by $\cup_\text{CR}$ the Chen-Ruan cup product on cohomology of the inertia stack.} $\text{deg}^{\text{Sym}}_0$, $\rho^\text{Sym}$, and $\mu^\text{Sym}$ on $H_\T^*(I\text{Sym}^n(\mathbb{C}^2))$ as follows. For $\phi\in H_{\T}^k(I\text{Sym}^n(\mathbb{C}^2))$, 
\begin{equation*}
\begin{split}
    &\text{deg}^{\text{Sym}}_0(\phi)= k\phi\, ,\\
    &\mu^\text{Sym}(\phi)=\left(\frac{\text{deg}_{\text{CR}}(\phi)}{2}-\frac{2n}{2}\right)\phi\, ,\\
    &\rho^\text{Sym}(\phi)= c_1^\T(\text{Sym}^n(\mathbb{C}^2))\cup_\text{CR}\phi\, .
    \end{split}    
\end{equation*}
There are {\em two} degree operators here: $\text{deg}_0^{\text{Sym}}$ extracts the usual degree of a cohomology class on the inertia orbifold, and  $\text{deg}_{\text{CR}}$ extracts the age-shifted degree. Also, we have $$\text{deg}_\text{CR}t_1=\text{deg}^{\text{Sym}}_0t_1=2=\text{deg}_\text{CR}t_2=\text{deg}^{\text{Sym}}_0t_2\, .$$

The multi-valued Givental space $\widetilde{\mathcal{H}}^\text{Sym}$ for $\text{Sym}^n(\mathbb{C}^2)$ is defined by 
$$\widetilde{\mathcal{H}}^\text{Sym}=H_\T^*(I\text{Sym}^n(\mathbb{C}^2))\otimes_{\mathbb{C}[t_1,t_2]}\mathbb{C}(t_1,t_2)[[\log z]]((z^{-1}))\, .$$

\begin{defn}
Let $\Psi^\text{\em Sym}: K_\T(\text{\em Sym}^n(\mathbb{C}^2))\to \widetilde{\mathcal{H}}^\text{\em Sym}$ be defined by 
\begin{equation*}
  \Psi^\text{\em Sym}(E)= z^{-\mu^\text{\em Sym}}z^{\rho^\text{\em Sym}}\left(\Gamma_\text{\em Sym}\cup (2\pi \sqrt{-1})^{\frac{\text{\em deg}^\text{\em Sym}_0}{2}}\widetilde{\text{\em ch}}(E)\right)\, ,
\end{equation*}
where $\widetilde{\text{\em ch}}(-)$ is the $\T$-equivariant orbifold Chern character, $\Gamma_\text{\em Sym}\in H_\T^*(I\text{\em Sym}^n(\mathbb{C}^2))$ is the $\T$-equivariant Gamma class of $\text{\em Sym}^n(\mathbb{C}^2)$ of \cite[Section 3.1]{cij}, and
the
operators $$
z^{-\mu^\text{\em Sym}}:\widetilde{\mathcal{H}}^\text{\em Sym}\to \widetilde{\mathcal{H}}^\text{\em Sym}\, ,
\ \ \ \ \
z^{\rho^\text{\em Sym}}:\widetilde{\mathcal{H}}^\text{\em Sym}\to \widetilde{\mathcal{H}}^\text{\em Sym}$$ 
are defined by 
$$z^{-\mu^\text{\em Sym}}=\sum_{k\geq 0}\frac{(-\mu^\text{\em Sym}\log z)^k}{k!}\, ,
\ \ \ \ \ 
z^{\rho^\text{\em Sym}}=\sum_{k\geq 0}\frac{(\rho^\text{\em Sym}\log z)^k}{k!}\, .$$
\end{defn}

The precise relationship between $\mathbb{FM}$ 
and $\mathsf{K}$ via Iritani's integral structure is 
the central result of the paper.

\begin{thm}\label{thm:comparison}
The following diagram is commutative{\footnote{Our variable $z$ corresponds to $-z$ in \cite{cij} 
 as can be seen by the difference in the quantum differential equation (\ref{eqn:qde_hilb_D}) here 
 and the quantum differential equation \cite[equation (2.5)]{cij}. After the substitution $z\mapsto -z$ in $\mathsf{K}$,  Theorem \ref{thm:comparison} matches
 the conventions of Iritani's framework in \cite{cij}.
}}:
\begin{displaymath}
    \xymatrix{ 
    K_\T(\hilbnc) \ar[r]^{\mathbb{FM}}\ar[d]_{\Psi^\text{\em Hilb}} & K_\T(\text{\em Sym}^n(\mathbb{C}^2))\ar[d]^{\Psi^\text{\em Sym}} \\
   {\widetilde{\mathcal{H}}^\text{\em Hilb}}\ar[r]^{\mathsf{C}\mathsf{K}\big|_{z\mapsto -z}}& {\widetilde{\mathcal{H}}^\text{\em Sym}}.}
\end{displaymath}
\end{thm}

The bottom row of the
diagram of Theorem  \ref{thm:comparison} is determined by  
the analytic continuation  of solutions of the
quantum differential equation of $\Hilb$ along the ray from
0 to $-1$ in the $q$-plane \cite[Theorem 4]{op29}.
A 
lifting of monodromies of the quantum differential equation of $\hilbnc$ to autoequivalences of $D^b_\T(\hilbnc)$ has been
announced  by  Bezrukavnikov and Okounkov in \cite[Sections 3.2.8 and 5.2.7]{o} and \cite[Section 3.2]{o2}.
In their upcoming
paper \cite{BezOk}, commutative diagrams parallel
to Theorem \ref{thm:comparison}  are constructed in cases
of {\em flops} of holomorphic symplectic manifolds.{\footnote{In fact, the study of commutative diagrams connecting derived
equivalences and the solutions of the quantum differential
equation has old roots in the subject. See, for example, \cite{BoriHorj, Horja}.
These papers refer to talks of Kontsevich on homological
mirror symmetry in the 1990s for the first formulations.}}
Theorem \ref{thm:comparison} fits into
the framework of \cite{BezOk} if the relationship between
$\Hilb$ and $\Sym(\CC^2)$ is viewed morally as a flop in their
sense.

A special aspect of the ray from $0$ to $-1$ is the identification
of the end result of the analytic continuation (the right side of the diagram)
with the orbifold geometry $\Sym(\CC^2)$. 
The identification of the
end results of other
paths from $0$ to $-1$
with geometric theories
is an interesting direction
of study. Are there twisted orbifold theories which
realize these analytic continuations?

%
    

\subsection{Acknowledgments}
We thank J.~Bryan, T.~Graber, Y.-P.~Lee, A.~Okounkov, and
Y. Ruan for many conversations about the crepant resolution
correspondence for 
$\Hilb$ and $\Sym(\CC^2)$. 
The paper began with Y. Jiang
asking us about the role of the Fourier-Mukai transformation
in the results of \cite{pt}.
We are very grateful to H. Iritani
for detailed discussions about his integral structure and
crepant resolution framework.

R.~P. was partially supported by 
 SNF-200020162928, ERC-2012-AdG-320368-MCSK, 
 ERC-2017-AdG-786580-MACI,
 Swiss\-MAP, and
the Einstein Stiftung. 
H.-H. ~T. was partially supported by NSF grant DMS-1506551.
The research presented here was
furthered during a visit of the authors to Humboldt University
in Berlin in June 2018. 

The project has received funding from the European Research Council (ERC) under the European Union Horizon 2020 Research and Innovation Program (grant No. 786580).

\section{Quantum differential equations}
\label{qqddee}

\subsection{The differential equation}

We recall the quantum differential equation for $\hilbnc$ calculated in \cite{op} and further studied in \cite{op29}. We follow here the exposition \cite{op,op29}.

The quantum differential equation (QDE) for the Hilbert schemes of points on $\mathbb{C}^2$ is
given by
\begin{equation}\label{qde}
q\frac{d}{dq}\Phi=\mathsf{M}_D\Phi\, , \quad \Phi\in \mathcal{F}\otimes_\mathbb{C} \mathbb{C}(t_1, t_2),
\end{equation}
where $\mathsf{M}_D$
is the operator of quantum multiplication by $D=-|2, 1^{n-2}\rangle$,
\begin{multline}\label{operator_M}
\mathsf{M}_D=
(t_1+t_2)\sum_{k>0}\frac{k}{2}\frac{(-q)^k+1}{(-q)^k-1}\alpha_{-k}\alpha_k
\, -\, \frac{t_1+t_2}{2}\frac{(-q)+1}{(-q)-1}|\cdot |
\\
+\frac{1}{2}\sum_{k,l>0}\Big[t_1t_2\alpha_{k+l}\alpha_{-k}\alpha_{-l}-\alpha_{-k-l}\alpha_k\alpha_l\Big]\, .
\end{multline}
Here $|\cdot|=\sum_{k>0}\alpha_{-k}\alpha_k$ is the energy operator. 

While the quantum differential equation
\eqref{qde} has a regular singular point at $q=0$, the point $q=-1$ is regular.

The quantum differential equation considered in Givental's theory contains a parameter $z$. In the case of the Hilbert schemes of points on $\mathbb{C}^2$, the QDE with parameter $z$ is 
\begin{equation}\label{qde_z}
zq\frac{d}{dq}\Phi=\mathsf{M}_D\Phi, \quad \Phi\in \mathcal{F}\otimes_\mathbb{C} \mathbb{C}(t_1, t_2)\, .
\end{equation}
For $\Phi\in \mathcal{F}\otimes_\mathbb{C}\mathbb{C}(t_1,t_2)$, define 
\begin{equation}\label{eqn:rescale}
\Phi_z=\Phi\left(\frac{t_1}{z}, \frac{t_2}{z}, q\right).
\end{equation}
Define $\Theta\in \text{Aut}(\mathcal{F})$ by $$\Theta|\mu\rangle=z^{\ell(\mu)}|\mu\rangle\, .$$
The following Proposition allows us to use the results in \cite{op29}.

\begin{prop}\label{ll44}
If $\Phi$ is a solution of (\ref{qde}), then $\Theta\Phi_z$ is a solution of (\ref{qde_z}).
\end{prop}

\noindent Proposition
\ref{ll44} follow immediately from the following direct computation.

\begin{lem}
For $k>0$, we have $\Theta\alpha_k=\frac{1}{z}\alpha_k\Theta$ and $\Theta\alpha_{-k}=z\alpha_{-k}\Theta$.
\end{lem}

\subsection{Solutions}

We recall the solution of QDE (\ref{qde}) constructed in \cite{op29}. Let $$J_\lambda\in \mathcal{F}\otimes_\com \com(t_1,t_2)$$ be the integral form of the Jack symmetric function depending on the parameter $\alpha=1/\theta$ of \cite{mac,op29}. Then 
\begin{equation}\label{eqn:jack}
\mathsf{J}^\lambda=t_2^{|\lambda|}t_1^{\ell(\cdot )}J_\lambda|_{\alpha=-t_1/t_2}
\end{equation}
is an eigenfunction of $\mathsf{M}_D(0)$ with eigenvalue $-c(\lambda;t_1,t_2):=-\sum_{(i,j)\in \lambda}[(j-1)t_1+(i-1)t_2]$. The coefficient of $$|\mu\rangle\in \mathcal{F}\otimes_\com \com(t_1,t_2)$$ in the expansion of $\mathsf{J}^\lambda$ is $(t_1t_2)^{\ell(\mu)}$ times a polynomial in $t_1$ and $t_2$ of degree $|\lambda|-\ell(\mu)$.

The paper \cite{op29} also uses a Hermitian pairing $\<-,-\>_{H}$ on the Fock space $\cF$ defined by the three following properties
\begin{enumerate}
\item[$\bullet$]
$\langle \mu|\nu\rangle_{H}=\frac{1}{(t_1t_2)^{\ell(\mu)}}\frac{\delta_{\mu\nu}}{\mathfrak{z}(\mu)}$,

\vspace{7pt}
\item[$\bullet$]
$\langle a f, g\rangle_H=a\langle f,g\rangle_H, \quad a\in \mathbb{C}(t_1, t_2)$,

\vspace{7pt}
\item[$\bullet$]
$\langle f,g\rangle_H=\overline
{\langle g,f\rangle}_H, \text{ where } \overline{a(t_1,t_2)}=a(-t_1, -t_2)$ .
\end{enumerate}

By a direct calculation, we find 
\begin{equation}\label{eqn:jack_pairings}
\<\mathsf{J}^\lambda, \mathsf{J}^\mu\>_H=\eta(\mathsf{J}^\lambda, \mathsf{J}^\mu)\, ,
\end{equation}
where $\eta$ is the $\T$-equivariant pairing on
$\Hilb$.
Since $\mathsf{J}^\lambda$ corresponds to the $\T$-equivariant class
of the $\T$-fixed point of $\Hilb$ associated to $\lambda$,
\begin{equation}\label{eqn:jack_lengths}
||\mathsf{J}^\lambda||^2=||\mathsf{J}^\lambda||_H^2=\prod_{\mathsf{w}: \text{ tangent weights at $\lambda$}} \mathsf{w}\,
\end{equation}
see \cite{op29}. 

There are solutions to (\ref{qde}) of the form $$\mathsf{Y}^\lambda(q)q^{-c(\lambda;t_1,t_2)}, \quad \mathsf{Y}^\lambda(q)\in \mathcal{F}\otimes_\mathbb{C}\mathbb{C}(t_1, t_2)[[q]],$$ which converge for $|q|<1$ and satisfy $\mathsf{Y}^\lambda(0)=\mathsf{J}^\lambda$. We refer to \cite[Chapter XIX]{ince} for a discussion of how these solutions are constructed.

By \cite[Corollary 1]{op29}, 
\begin{equation}\label{eqn:length_Y}
\langle \mathsf{Y}^\lambda(q), \mathsf{Y}^\mu(q)\rangle_H=\delta_{\lambda\mu} ||\mathsf{J}^\lambda||_H^2=\langle \mathsf{J}^\lambda, \mathsf{J}^\mu\rangle_H.
\end{equation}

As in \cite[Section 3.1.3]{op29}, let $\mathsf{Y}$ be the matrix whose column vectors are $\mathsf{Y}^\lambda$. 
Fix an auxiliary basis $\{e_\lambda\}$ of $\cF$. We then view $\mathsf{Y}$ as the matrix representation\footnote{In the domain of $\mathsf{Y}$ we use the basis $\{e_\lambda\}$, while in the range of $\mathsf{Y}$ we use the basis $\{|\mu\rangle\}$.} of an operator such that $\mathsf{Y}(e_\lambda)=\mathsf{Y}^\lambda$.

Define the following further diagonal matrices in the basis $\{e_\lambda\}$:
\begin{center}
\begin{tabular}[c]{|l|r|}
\hline
{Matrix} & {Eigenvalues} \\
\hline
$L$ & $z^{-|\lambda|}\prod_{\mathsf{w}: \text{ tangent weights at $\lambda$}} \mathsf{w}^{1/2}$\\
\hline
$L_0$ & $q^{-c(\lambda;t_1,t_2)/z}$\\
\hline
\end{tabular}
\end{center}

Define
$$\mathsf{Y}_z=\mathsf{Y}\left(\frac{t_1}{z}, \frac{t_2}{z}, q\right).$$ 
Consider the following solution to (\ref{qde_z}),
\begin{equation}\label{eqn:sol_des}
\mathsf{S}=\Theta\mathsf{Y}_zL^{-1}L_0\, .
\end{equation}
We may view $\mathsf{S}$ as the matrix representation of an operator where in the domain  we use the  basis $\{e_\lambda\}$ while in the range  we use the basis $\{|\mu\rangle\}$.

\begin{prop}\label{prop:qde_soln1}
$\Theta\mathsf{Y}_zL^{-1}$ can be expanded into a convergent power series in $1/z$ with coefficients $\text{End}(\mathcal{F})$-valued analytic functions in $q, t_1, t_2$.
\end{prop}
\begin{proof}
Let $\Phi^\lambda$ be the column of $\Theta\mathsf{Y}_zL^{-1}$ indexed by $\lambda$. By construction of $\mathsf{Y}$, 
\begin{equation*}
 \Theta\mathsf{Y}_zL^{-1}\Big|_{q=0}=\Theta\mathsf{J}_zL^{-1}, 
\end{equation*}
hence $\Phi^\lambda\Big|_{q=0}=\Theta\mathsf{J}_z^\lambda z^{|\lambda|}\prod_{\mathsf{w}: \text{ tangent weights at }\lambda} \mathsf{w}^{-1/2}$. 
Write $\mathsf{J}^\lambda=\sum_\epsilon \mathsf{J}^\lambda_\epsilon(t_1, t_2) |\epsilon\rangle$. Then we have 
\begin{equation*}
\begin{split}
\Theta\mathsf{J}_z^\lambda z^{|\lambda|}
&=\sum_\epsilon \mathsf{J}^\lambda_\epsilon(t_1/z, t_2/z)z^{\ell(\epsilon)}z^{|\lambda|} |\epsilon\rangle\\
&=\sum_\epsilon \mathsf{J}^\lambda_\epsilon(t_1, t_2)z^{-2\ell(\epsilon)}z^{\ell(\epsilon)-|\lambda|}z^{\ell(\epsilon)}z^{|\lambda|} |\epsilon\rangle=\mathsf{J}^\lambda.
\end{split}
\end{equation*}
Together with (\ref{eqn:jack_lengths}), we find $\Phi^\lambda\Big|_{q=0}=\mathsf{J}^\lambda/||\mathsf{J}^\lambda||$.

Since  $\mathsf{S}$ is a solution to (\ref{qde_z}), $\Phi^\lambda$ is a solution to the differential equation 
\begin{equation}\label{qde_l}
zq\frac{d}{dq}\Phi^\lambda=(\mathsf{M}_D+c(\lambda; t_1, t_2))\Phi^\lambda.
\end{equation} 
By uniqueness of solutions to (\ref{qde_l}) with given initial conditions, $\Phi^\lambda$ can also be constructed using the Peano-Baker series (see  \cite{bs}) with the initial condition $$\Phi^\lambda\Big|_{q=0}=\mathsf{J}^\lambda/||\mathsf{J}^\lambda||\,. $$ As the Peano-Baker series is manifestly a power series in $z^{-1}$ with analytic coefficients, the Proposition follows.
\end{proof}

\section{Descendent Gromov-Witten theory}

\subsection{Hilbert schemes} \label{HSHS}
Let $\mathsf{S}^\text{Hilb}(q,t_D)$ be the generating series of genus $0$ descendent Gromov-Witten invariants of $\hilbnc$,
\begin{equation}\label{eqn:S_hilb}
\eta(a, \mathsf{S}^\text{Hilb}(q,t_D)b)=\eta (a,b)+\sum_{k\geq 0}z^{-1-k}\sum_{m, d}\frac{q^d}{m!}\langle a, \underbrace{t_D D,...,t_D D}_{m}, b\psi_{m+2}^k\rangle_{0,d}^{\hilbnc}    
\end{equation}
By definition, $\mathsf{S}^\text{Hilb}$ is a formal power series in $1/z$ whose coefficients are in $\text{End}(\mathcal{F})[t_D][[q]]$, written in the basis $\{|\mu\rangle\}$. $\mathsf{S}^\text{Hilb}(q,t_D)$ satisfies the following two differential equations:
\begin{equation}\label{eqn:qde_hilb_D}
z\frac{\partial}{\partial t_D}\mathsf{S}^\text{Hilb}(q,t_D)=(D\star_{t_D})\mathsf{S}^\text{Hilb}(q,t_D),    
\end{equation}
\begin{equation}\label{eqn:div_hilb}
zq\frac{\partial}{\partial q}\mathsf{S}^\text{Hilb}(q,t_D)-z\frac{\partial}{\partial t_D}\mathsf{S}^\text{Hilb}(q,t_D)=-\mathsf{S}^\text{Hilb}(q,t_D)(D\cdot).
\end{equation}
Here $(D\star_{t_D})=(D\star_{t_D D})$ is the operator of quantum multiplication by the divisor $D$ at the point\footnote{We use $t_D$ to denote the coordinate of $D$.} $t_D D$, 
\begin{equation*}
\eta((D\star_{t_D})a, b)=\sum_{m\geq 0, d\geq 0}\frac{q^d}{m!}\langle D, a, \underbrace{t_D D,...,t_D D}_{m}, b \rangle_{0,d}^{\hilbnc},    
\end{equation*}
and $(D\cdot)$ is the operator of classical cup product by $D$. In particular,
\begin{equation}\label{eqn:oper_D_0}
 (D\star_{t_D})\Big|_{t_D=0}=\mathsf{M}_D(q), \quad (D\cdot)=(D\cdot)\Big|_{t_D=0}=\mathsf{M}_D(0)\, .
\end{equation}
Equation (\ref{eqn:qde_hilb_D}) follows from the topological recursion relations in genus $0$. Equation (\ref{eqn:div_hilb}) follows from the divisor equations for {\em descendent} Gromov-Witten invariants.

We first determine $\mathsf{S}^\text{Hilb}\Big|_{t_D=0}$. Combining (\ref{eqn:qde_hilb_D}) and (\ref{eqn:div_hilb}) and setting $t_D=0$, we find 
\begin{equation*}
zq\frac{\partial}{\partial q}\left(\mathsf{S}^\text{Hilb}\Big|_{t_D=0}\right)=\mathsf{M}_D(q)\left(\mathsf{S}^\text{Hilb}\Big|_{t_D=0}\right)-\left(\mathsf{S}^\text{Hilb}\Big|_{t_D=0}\right)\mathsf{M}_D(0)\, .    
\end{equation*}
So, we see
\begin{equation*}
\begin{split}
zq\frac{\partial}{\partial q}\left(\mathsf{S}^\text{Hilb}\Big|_{t_D=0}\mathsf{J}^\lambda/||\mathsf{J}^\lambda||\right)&=\mathsf{M}_D(q)\left(\mathsf{S}^\text{Hilb}\Big|_{t_D=0}\mathsf{J}^\lambda/||\mathsf{J}^\lambda||\right)-\left(\mathsf{S}^\text{Hilb}\Big|_{t_D=0}\right)\mathsf{M}_D(0)\mathsf{J}^\lambda/||\mathsf{J}^\lambda||\\
&=\mathsf{M}_D(q)\left(\mathsf{S}^\text{Hilb}\Big|_{t_D=0}\mathsf{J}^\lambda/||\mathsf{J}^\lambda||\right)+c(\lambda;t_1,t_2)\left(\mathsf{S}^\text{Hilb}\Big|_{t_D=0}\mathsf{J}^\lambda/||\mathsf{J}^\lambda||\right).    
\end{split}
\end{equation*}
Since $\mathsf{S}^\text{Hilb}\Big|_{t_D=0, q=0}=\text{Id}$, we have $\left(\mathsf{S}^\text{Hilb}\Big|_{t_D=0}\mathsf{J}^\lambda/||\mathsf{J}^\lambda||\right)\Big|_{q=0}=\mathsf{J}^\lambda/||\mathsf{J}^\lambda||$. Comparing the result with the proof of Proposition \ref{prop:qde_soln1}, we conclude 
\begin{equation*}
 \mathsf{S}^\text{Hilb}\Big|_{t_D=0}\mathsf{J}^\lambda/||\mathsf{J}^\lambda||=\Phi^\lambda,   
\end{equation*}
as $\cF$-valued power series.

Let $\mathsf{A}:\cF\to \cF$ be defined by $\mathsf{A}(e_\lambda)=\mathsf{J}^\lambda/||\mathsf{J}^\lambda||$ . The above discussion
yields the following result.

\begin{prop}\label{prop:S_hilb0}
As power series in $1/z$, we have $\mathsf{S}^\text{Hilb}\Big|_{t_D=0}\mathsf{A}=\mathsf{S}L_0^{-1}$. 
\end{prop}

\noindent By definition, $\mathsf{S}^\text{Hilb}$ is a formal power series in $q$. By Proposition \ref{prop:S_hilb0},  $\mathsf{S}^\text{Hilb}$ is analytic in $q$.

By the divisor equation for primary Gromov-Witten invariants, we have $$q\frac{\partial}{\partial q}(D\star_{t_D})-\frac{\partial}{\partial t_D}(D\star_{t_D})=0\, .$$
A direct calculation then 
shows that the two differential operators $$z\frac{\partial}{\partial t_D}-(D\star_{t_D})\ \ \  \text{and} \ \ \  zq\frac{\partial}{\partial q}-z\frac{\partial}{\partial t_D}-(-)(D\cdot)$$ commute. Therefore, equation (\ref{eqn:qde_hilb_D}) and Proposition \ref{prop:S_hilb0} uniquely determine $\mathsf{S}^\text{Hilb}(q,t_D)$.

\subsection{Symmetric products}
We introduce another copy of the Fock space $\cF$ which we denote by $\widetilde{\cF}$. An additive isomorphism $$\widetilde{\cF}\otimes_\mathbb{C}\mathbb{C}[t_1, t_2]\simeq \bigoplus_{n\geq 0}H_\T^*(I\text{Sym}^n(\mathbb{C}^2), \mathbb{C})\, ,$$ is given by identifying $|{\mu}\rangle\in \widetilde{\cF}$ with the fundamental class $[I_\mu]$ of the component of the inertia orbifold $I\text{Sym}^n(\mathbb{C}^2)$ indexed by $\mu$. The orbifold Poincar\'e pairing $(-,-)^\text{Sym}$ induces via this identification a pairing on $\widetilde{\cF}$, $$\widetilde{\eta}(\mu, \nu)=\frac{1}{(t_1t_2)^{\ell(\mu)}}\frac{\delta_{\mu\nu}}{\mathfrak{z}(\mu)}.$$

Following \cite[Equation (1.6)]{pt}, we define $$|\widetilde{\mu}\rangle=(-\sqrt{-1})^{\ell(\mu)-|\mu|}|\mu\rangle\in \widetilde{\cF}.$$ 
We will use the following linear isomorphism 
\begin{equation}\label{eqn:F_isom}
\mathsf{C}: \cF\to \widetilde{\cF}\, , \ \ \quad |\mu\rangle\mapsto |\widetilde{\mu}\rangle\, ,
\end{equation}
which is compatible with the pairings $\eta$ and $\widetilde{\eta}$.

We recall the definition of the ramified Gromov-Witten invariants of $\text{Sym}^n(\mathbb{C}^2)$ following \cite[Section 3.2]{pt}. Consider the moduli space  $\overline{\mathcal{M}}_{g,r+b}(\text{Sym}^n(\mathbb{C}^2))$ 
of stable maps to $\text{Sym}^n(\mathbb{C}^2)$ and let $$\overline{\mathcal{M}}_{g,r,b}(\text{Sym}^n(\mathbb{C}^2))=[\left(ev_{r+1}^{-1}(I_{(2)})\cap...\cap ev_{r+b}^{-1}(I_{(2)}) \right)/\Sigma_b]$$ where the symmetric group $\Sigma_b$ acts by permuting the last $b$ marked points. Define ramified descendent Gromov-Witten invariants by 
\begin{equation*}
\left\langle \prod_{i=1}^r I_{\mu^i}\psi^{k_i}\right\rangle_{g,b}^{\text{Sym}^n(\mathbb{C}^2)}=\int_{[\overline{\mathcal{M}}_{g,r,b}(\text{Sym}^n(\mathbb{C}^2))]^{vir}}\prod_{i=1}^r ev_i^*([I_{\mu^i}])\psi^{k_i}\, .    \end{equation*}

Let $\mathsf{S}^\text{Sym}(u, \tilde{t})$ be the generating function of genus $0$ ramified descendent Gromov-Witten invariants of $\text{Sym}^n(\mathbb{C}^2)$,

\begin{equation}
\tilde{\eta}(a, \mathsf{S}^\text{Sym}(u, \tilde{t})b)=\tilde{\eta}(a,b)+\sum_{k\geq 0}z^{-1-k}\sum_{m,d}\frac{u^d}{m!}\langle a, \underbrace{\tilde{t}I_{(2)},..., \tilde{t} I_{(2)}}_{m}, b\psi_{m+2}^k \rangle_{0,d}^{\text{Sym}^n(\mathbb{C}^2)}.    
\end{equation}
By definition, $\mathsf{S}^\text{Sym}$ is a formal power series in $1/z$ whose coefficients are in $\text{End}(\widetilde{\mathcal{F}})[\tilde{t}][[u]]$, written in the basis $\{|\widetilde{\mu}\rangle\}$. $\mathsf{S}^\text{Sym}$ satisfies the following two differential equations:
\begin{equation}\label{eqn:qde_sym_D}
z\frac{\partial}{\partial \tilde{t}}\mathsf{S}^\text{Sym}(u, \tilde{t})=(I_{(2)}\star_{\tilde{t}})\mathsf{S}^\text{Sym}(u, \tilde{t})\, ,
\end{equation}
\begin{equation}\label{eqn:div_sym}
\frac{\partial}{\partial u}\mathsf{S}^\text{Sym}(u, \tilde{t})=\frac{\partial}{\partial \tilde{t}}\mathsf{S}^\text{Sym}(u, \tilde{t})\, .    
\end{equation}
Here $(I_{(2)}\star_{\tilde{t}})=(I_{(2)}\star_{\tilde{t}I_{(2)}})$ is the operator of quantum multiplication by the divisor $I_{(2)}$ at the point $\tilde{t}I_{(2)}$,
\begin{equation*}
\tilde{\eta}((I_{(2)}\star_{\tilde{t}})a, b)=\sum_{m, d}\frac{u^d}{m!}\langle I_{(2)}, a, \underbrace{\tilde{t}I_{(2)},..., \tilde{t} I_{(2)}}_{m}, b \rangle_{0,d}^{\text{Sym}^n(\mathbb{C}^2)}\, .     
\end{equation*}
Equation (\ref{eqn:qde_sym_D}) follows from the genus $0$ topological recursion relations for orbifold Gromov-Witten invariants, see  \cite{Tseng}. Equation (\ref{eqn:div_sym}) follows from divisor equations for {\em ramified} orbifold Gromov-Witten invariants, see \cite{bg}.

We first compare the operators $(D\star_{t_D D})$ and $(I_{(2)}\star_{\tilde{t}I_{(2)}})$. For simplicity, write $(2)$ for the partition $(2, 1^{n-2})$. By \cite[Theorem 4]{pt}, we have 
\begin{equation*}
\begin{split}
\langle D,\underbrace{D,...,D}_{k}, \lambda, \mu\rangle^\text{Hilb}
=&(-1)^{k+1}\langle (2),\underbrace{(2),...,(2)}_{k}, \lambda, \mu\rangle^\text{Hilb}\\
=&(-1)^{k+1}\langle (\tilde{2}),\underbrace{(\tilde{2}),...,(\tilde{2})}_{k}, \tilde{\lambda}, \tilde{\mu}\rangle^\text{Sym}\\
=&\langle -(\tilde{2}),\underbrace{-(\tilde{2}),...,-(\tilde{2})}_{k}, \tilde{\lambda}, \tilde{\mu}\rangle^\text{Sym},
\end{split}
\end{equation*}
where $(\tilde{-})$ is defined in \cite[Equation (1.6)]{pt}. Therefore, under the identification $|\mu\rangle \mapsto |\tilde{\mu}\rangle$, we have 
\begin{equation}
D\star_{t_D D}=-(\tilde{2})\star_{t_D(-(\tilde{2}))}.    
\end{equation}
Now, $$(\tilde{2})=(-i)^{n-1-n}I_{(2)}=(-i)^{-1}I_{(2)}=iI_{(2)}\, .$$
Hence we have, after $-q=e^{iu}$,  
\begin{equation}\label{eqn:oper_D_comp}
D\star_{t_D D}=(-i)I_{(2)}\star_{\tilde{t}I_{(2)}}, \quad \tilde{t}=(-i)t_D\, .    
\end{equation}

Consider now $\mathsf{S}^\text{Sym}\Big|_{\tilde{t}=0}$. By (\ref{eqn:qde_sym_D}) and (\ref{eqn:div_sym}), we have 
\begin{equation*}
z\frac{\partial}{\partial u}\mathsf{S}^\text{Sym}(u, \tilde{t})=(I_{(2)}\star_{\tilde{t}})\mathsf{S}^\text{Sym}(u, \tilde{t})\, .  \end{equation*}
Setting $\tilde{t}=0$ and using (\ref{eqn:oper_D_0}) and (\ref{eqn:oper_D_comp}), we find 
\begin{equation*}
z\frac{\partial}{\partial u}\left(\mathsf{S}^\text{Sym}\Big|_{\tilde{t}=0}\right)=i\mathsf{M}_D(-e^{iu})\left(\mathsf{S}^\text{Sym}\Big|_{\tilde{t}=0}\right)\, .    
\end{equation*}
Since $\frac{\partial}{\partial u}=iq\frac{\partial}{\partial q}$, we find that, after $-q=e^{iu}$,
\begin{equation}
zq\frac{\partial}{\partial q}\left(\mathsf{S}^\text{Sym}\Big|_{\tilde{t}=0}\right)=\mathsf{M}_D(q)\left(\mathsf{S}^\text{Sym}\Big|_{\tilde{t}=0}\right)\, .    
\end{equation}

Recall  $\mathsf{S}=\Theta\mathsf{Y}_zL^{-1}L_0$ also satisfied the same equation. We may then compare $\Theta\mathsf{Y}_zL^{-1}L_0$ and $\left(\mathsf{S}^\text{Sym}\Big|_{\tilde{t}=0}\right)$ by comparing them at $u=0$ which corresponds to $q=-1$. Set
$$B=\mathsf{S}\Big|_{q=-1}=\Theta\mathsf{Y}_zL^{-1}L_0\Big|_{q=-1}\, .$$ Since $\mathsf{S}^\text{Sym}\Big|_{\tilde{t}=0, u=0}=\text{Id}$, we have, after $-q=e^{iu}$, 

\begin{equation}\label{eqn:hilb_sym_0}
\mathsf{S}^\text{Sym}\Big|_{\tilde{t}=0}=\mathsf{C}\mathsf{S}B^{-1}\mathsf{C}^{-1}\, .
\end{equation}
By Proposition \ref{prop:S_hilb0}, we have 
\begin{equation}\label{eqn:hilb_sym_0.5}
\mathsf{C}\mathsf{S}B^{-1}\mathsf{C}^{-1}=\mathsf{C}\mathsf{S}^\text{Hilb}\Big|_{t_D=0}\mathsf{A}L_0B^{-1}\mathsf{C}^{-1}\, .
\end{equation}
Since $\mathsf{A}L_0\mathsf{A}^{-1}=q^{D/z}$, $$\mathsf{A}L_0B^{-1}=\mathsf{A}L_0\mathsf{A}^{-1}\mathsf{A}B^{-1}=q^{D/z}\mathsf{A}B^{-1}.$$
Define $\mathsf{K}=B\mathsf{A}^{-1}$.  We can then rewrite (\ref{eqn:hilb_sym_0.5}) as
\begin{equation}\label{eqn:hilb_sym_0.75}
\mathsf{S}^\text{Sym}\Big|_{\tilde{t}=0}=\mathsf{C}\mathsf{S}^\text{Hilb}\Big|_{t_D=0}q^{D/z}\mathsf{K}^{-1}\mathsf{C}^{-1}\, .
\end{equation}

By the divisor equation for orbifold Gromov-Witten invariants in \cite{bg} (see also \cite[Section 3.2]{pt}), we have $$\frac{\partial}{\partial u}(I_{(2)}\star_{\tilde{t}})-\frac{\partial}{\partial \tilde{t}}(I_{(2)}\star_{\tilde{t}})=0\,  .$$ 
A direct calculation then shows that the two differential operators $$z\frac{\partial}{\partial \tilde{t}}-(I_{(2)}\star_{\tilde{t}})\ \ \  \text{and}\ \ \  \frac{\partial}{\partial u}-\frac{\partial}{\partial \tilde{t}}$$ commute. Therefore $\mathsf{S}^\text{Sym}(u, \tilde{t})$ is uniquely determined by equation (\ref{eqn:qde_sym_D}) and $\mathsf{S}^\text{Sym}\Big|_{\tilde{t}=0}$. By (\ref{eqn:oper_D_comp}), we have 
\begin{equation*}
z\frac{\partial}{\partial t_D}-(D\star_{t_D})=i\left(z\frac{\partial}{\partial \tilde{t}}-(I_{(2)}\star_{\tilde{t}})) \right)\,, 
\end{equation*}
after $-q=e^{iu}$.    
Then equation
(\ref{eqn:hilb_sym_0.75})  implies the following result.
\begin{thm}\label{thm:S_matrices}
After $-q=e^{iu}$ and $\tilde{t}=(-i)t_D$, we have 
\begin{equation*}
\mathsf{S}^\text{\em Sym}(u, \tilde{t})=\mathsf{C}\mathsf{S}^\text{\em Hilb}(q, t_D)q^{D/z}\mathsf{K}^{-1}\mathsf{C}^{-1}.    
\end{equation*}
\end{thm}

\subsection{Proof of Theorem \ref{kkkk}}
By the definition of $B$ and Proposition \ref{prop:qde_soln1}, $\mathsf{K}$ is an $\text{End}(\mathcal{F})$-valued power series in $1/z$ of the form $$\mathsf{K}=\text{Id}+O(1/z)\, .$$ By Theorem \ref{thm:S_matrices} and the fact that $\mathsf{S}^\text{Hilb}$ and $\mathsf{S}^\text{Sym}$ are symplectic, it follows that $\mathsf{K}$ is also symplectic. 

Next, we explicitly evaluate $\mathsf{K}$. By the definition of $B$ and \cite[Theorem 4]{op29}, we have 
\begin{equation}
\begin{split}
B&=\left(\Theta\mathsf{Y}_zL^{-1}L_0\right)\Big|_{q=-1}\\
&=\frac{1}{(2\pi\sqrt{-1})^{|\cdot|}}\Theta \mathbf{\Gamma}_z\mathsf{H}_z \left(\mathsf{G}_{\text{DT}z}^{-1}L_0 \right)\Big|_{q=-1}L^{-1}\, .
\end{split}
\end{equation}
Here $|\cdot|=\sum_{k>0}\alpha_{-k}\alpha_k$ is the energy operator. $\mathsf{G}_\text{DT}$ is the diagonal matrix in the basis $\{e_\lambda\}$ with eigenvalues $$q^{-c(\lambda;t_1,t_2)}\prod_{\mathsf{w}: \text{ tangent weights at $\lambda$}}\frac{1}{\Gamma(\mathsf{w}+1)}\, ,$$ see \cite[Section 3.1.2]{op29}. The operator $\mathbf{\Gamma}$ is given by $$\mathbf{\Gamma}|\mu\rangle=\frac{(2\pi\sqrt{-1})^{\ell(\mu)}}{\prod_i \mu_i}\mathsf{G}_{\text{GW}}(t_1,t_2)|\mu\rangle\, ,$$ see \cite[Section 3.3]{op29}, where
 $$\mathsf{G}_{\text{GW}}(t_1,t_2)|\mu\rangle=\prod_ig(\mu_i, t_1)g(\mu_i, t_2)|\mu\rangle\, ,$$   and $$g(\mu_i, t_1)g(\mu_i, t_2)=\frac{\mu_i^{\mu_it_1}\mu_i^{\mu_i t_2}}{\Gamma(\mu_i t_1)\Gamma(\mu_i t_2)}\, ,$$ see \cite[Section 3.1.2]{op29}. 
Define
$$\mathbf{\Gamma}_z=\mathbf{\Gamma}\left(\frac{t_1}{z}, \frac{t_2}{z}\right).$$

Since 
\begin{equation*}
\mathsf{K}=B\mathsf{A}^{-1}=\frac{1}{(2\pi\sqrt{-1})^{|\cdot|}}\Theta \mathbf{\Gamma}_z\mathsf{H}_z \left(\mathsf{G}_{\text{DT}z}^{-1}L_0 \right)\Big|_{q=-1}L^{-1}\mathsf{A}^{-1},    
\end{equation*}
and $||\mathsf{J}^\lambda||=\prod_{\mathsf{w}: \text{ tangent weights at $\lambda$}}{\mathsf{w}^{1/2}}$, we see that $\mathsf{K}$ is the operator given by 
\begin{equation}\label{prop:op_K}
\mathsf{K}(\mathsf{J}^\lambda)= \frac{z^{|\lambda|}}{(2\pi\sqrt{-1})^{|\lambda|}}\prod_{\mathsf{w}: \text{ tangent weights at $\lambda$}}{\Gamma(\mathsf{w}/z+1)}\Theta \mathbf{\Gamma}_z\mathsf{H}^\lambda_z\, .   
\end{equation}
The proof Theorem \ref{kkkk} is
complete. \qed

\section{Descendent correspondence}
\subsection{Variables}
We compare the descendent Gromov-Witten theories of $\hilbnc$ and $\text{Sym}^n(\mathbb{C}^2)$. The following identifications will be used throughout:
\begin{equation}
-q=e^{iu}\, , \ \ \ \ \tilde{t}=(-i)t_D\, .
\end{equation}

\subsection{Genus $0$}\label{sec:genus0}
Following \cite{g}, consider the Givental spaces
\begin{equation*}
\begin{split}
&\mathcal{H}^\text{Hilb}=H_\T^*(\hilbnc)\otimes_{\mathbb{C}[t_1,t_2]}\mathbb{C}(t_1,t_2)[[q]]((z^{-1}))\, ,\\
&\mathcal{H}^\text{Sym}=H_\T^*(\text{Sym}^n(\mathbb{C}^2))\otimes_{\mathbb{C}[t_1,t_2]}\mathbb{C}(t_1,t_2)[[u]]((z^{-1}))\, ,
\end{split}    
\end{equation*}
equipped with the symplectic forms
\begin{equation*}
\begin{split}
&(f, g)^{\mathcal{H}^\text{Hilb}}=\text{Res}_{z=0} (f(-z),g(z))^\text{Hilb}\,, \  \quad f,g\in \mathcal{H}^\text{Hilb}\, ,\\
&(f,g)^{\mathcal{H}^\text{Sym}}=\text{Res}_{z=0} (f(-z), g(z))^\text{Sym}\, ,\  \quad f,g\in \mathcal{H}^\text{Sym}\, .
\end{split}    
\end{equation*}
The choice of bases $$\{|\mu\rangle \big| \mu\in \text{Part}(n)\}\subset H_\T^*(\hilbnc)\, , \ \ \quad \{|\widetilde{\mu}\rangle \big| \mu\in \text{Part}(n)\}\subset H_\T^*(\text{Sym}^n(\mathbb{C}^2))\, ,$$
yields Darboux coordinate systems $\{p_a^\mu, q_b^\nu\}$, $\{\widetilde{p}_a^\mu, \widetilde{q}_b^\nu \}$. General points of $\mathcal{H}^\text{Hilb}, \mathcal{H}^\text{Sym}$ can be written in the form
\begin{equation*}
\begin{split}
&\underbrace{\sum_{a\geq 0}\sum_\mu p_a^\mu |\mu\rangle\frac{(t_1t_2)^{\ell(\mu)}\mathfrak{z}(\mu)}{(-1)^{|\mu|-\ell(\mu)}}(-z)^{-a-1}}_{\mathbf{p}}+\underbrace{\sum_{b\geq 0}\sum_\nu q_b^\nu|\nu\rangle z^b}_{\mathbf{q}}\in \mathcal{H}^\text{Hilb}\, ,\\
&\underbrace{\sum_{a\geq 0}\sum_\mu \widetilde{p}_a^\mu |\widetilde{\mu}\rangle\frac{(t_1t_2)^{\ell(\mu)}\mathfrak{z}(\mu)}{1}(-z)^{-a-1}}_{\widetilde{\mathbf{p}}}+\underbrace{\sum_{b\geq 0}\sum_\nu \widetilde{q}_b^\nu|\widetilde{\nu}\rangle z^b}_{\widetilde{\mathbf{q}}}\in \mathcal{H}^\text{Sym}\, .
\end{split}    
\end{equation*}

Define
the Lagrangian cones associated to the generating functions of genus $0$ descendent and ancestor Gromov-Witten invariants as follows:
\begin{equation*}
\begin{split}
&\mathcal{L}^\text{Hilb}=\{(\mathbf{p}\, , \mathbf{q})\big|\mathbf{p}=d_\mathbf{q}\cF^\text{Hilb}_0 \}\subset \mathcal{H}^\text{Hilb}\, ,\quad  \mathcal{L}^\text{Hilb}_{an, t_D}=\{(\mathbf{p}, \mathbf{q})\big|\mathbf{p}=d_\mathbf{q}\cF^\text{Hilb}_{an, t_D,0}\}\subset \mathcal{H}^\text{Hilb}\, ,\\
&\mathcal{L}^\text{Sym}=\{(\widetilde{\mathbf{p}}, \widetilde{\mathbf{q}}) \big|\widetilde{\mathbf{p}}=d_{\widetilde{\mathbf{q}}}\cF^\text{Sym}_0 \}\subset \mathcal{H}^\text{Sym}\, ,\quad \mathcal{L}^\text{Sym}_{an, \tilde{t}}=\{(\widetilde{\mathbf{p}}, \widetilde{\mathbf{q}}) \big|\widetilde{\mathbf{p}}=d_{\widetilde{\mathbf{q}}}\cF^\text{Sym}_{a, \tilde{t},0}  \} \subset \mathcal{H}^\text{Sym}\, ,  \end{split}   
\end{equation*}
where 
\begin{equation*}
\begin{split}
\cF^\text{Hilb}_0(\mathbf{t})=\sum_{d,k\geq 0} \frac{q^d}{k!} \langle\underbrace{\mathbf{t}(\psi),...,\mathbf{t}(\psi)}_{k} \rangle_{0,d}^\text{Hilb}\, , \quad \cF^\text{Hilb}_{an, t_D,0}(\mathbf{t})=\sum_{d,k, l\geq 0} \frac{q^d}{k!l!} \langle\underbrace{\mathbf{t}(\bar{\psi}),...,\mathbf{t}(\bar{\psi})}_{k}, \underbrace{t_D D,..., t_D D}_{l} \rangle_{0,d}^\text{Hilb}\, ,\\
\cF^\text{Sym}_0(\widetilde{\mathbf{t}})=\sum_{b,k\geq 0} \frac{u^b}{k!} \langle\underbrace{\widetilde{\mathbf{t}}(\psi),...,\widetilde{\mathbf{t}}(\psi)}_{k} \rangle_{0,b}^\text{Sym}\, , \quad \cF^\text{Sym}_{an, \tilde{t},0}(\widetilde{\mathbf{t}})=\sum_{b,k, l\geq 0} \frac{u^b}{k!l!} \langle\underbrace{\widetilde{\mathbf{t}}(\bar{\psi}),...,\widetilde{\mathbf{t}}(\bar{\psi})}_{k}, \underbrace{t I_{(2)},..., t I_{(2)}}_{l} \rangle_{0,b}^\text{Sym}\, .
\end{split}    
\end{equation*}
Here, $\mathbf{q}=\mathbf{t}-1z$ and $\widetilde{\mathbf{q}}=\widetilde{\mathbf{t}}-1z$ are dilaton shifts.

By the descendent/ancestor relations \cite{cg}, we have 
\begin{equation*}
\mathcal{L}^\text{Hilb}=\mathsf{S}^\text{Hilb}(q, t_D)^{-1}\mathcal{L}^\text{Hilb}_{an, t_D}\,, \  \quad  \mathcal{L}^\text{Sym}=\mathsf{S}^\text{Sym}(u, \tilde{t})^{-1}\mathcal{L}^\text{Sym}_{an, \tilde{t}}\, .     
\end{equation*}
By the genus $0$ crepant
resolution correspondence proven\footnote{In particular, the results of \cite{bg} implies that $\mathcal{L}^\text{Hilb}_{an, t_D}$ is analytic in $q$.} in \cite{bg}, we have 
\begin{equation*}
\mathsf{C}\mathcal{L}^\text{Hilb}_{an, t_D}= \mathcal{L}^\text{Sym}_{an, \tilde{t}}\, .   
\end{equation*}

\begin{thm}\label{thm:genus0_corresp} We have
$\mathcal{L}^\text{Sym}=\mathsf{C}\mathsf{K}q^{-D/z}\mathcal{L}^\text{Hilb}$.
\end{thm}
\begin{proof}
Using Theorem \ref{thm:S_matrices}, we calculate
\begin{equation*}
\begin{split}
\mathcal{L}^\text{Sym}
=&\mathsf{S}^\text{Sym}(u, \tilde{t})^{-1}\mathcal{L}^\text{Sym}_{an, \tilde{t}}\\
=&\mathsf{S}^\text{Sym}(u, \tilde{t})^{-1}\mathsf{C}\mathcal{L}^\text{Hilb}_{an, t_D}\\
=&\mathsf{C}\mathsf{K}q^{-D/z}\mathsf{S}^\text{Hilb}(q, t_D)^{-1}\mathcal{L}^\text{Hilb}_{an, t_D}\\
=&\mathsf{C}\mathsf{K}q^{-D/z}\mathcal{L}^\text{Hilb}\, . 
\end{split}    
\end{equation*}
\end{proof}

\subsection{Higher genus}\label{sec:higher_genus}
Consider the total descendent potentials, 
\begin{equation*}
\begin{split}
&\mathcal{D}^\text{Hilb}=\exp\left(\sum_{g\geq 0}\hbar^{g-1}\cF_g^\text{Hilb} \right)\, , \quad \cF^\text{Hilb}_g(\mathbf{t})=\sum_{d,k\geq 0} \frac{q^d}{k!} \langle\underbrace{\mathbf{t}(\psi),...,\mathbf{t}(\psi)}_{k} \rangle_{g,d}^\text{Hilb}\, ,\\ 
&\mathcal{D}^\text{Sym}=\exp\left(\sum_{g\geq 0}\hbar^{g-1}\cF_g^\text{Sym} \right)\, , \quad \cF^\text{Sym}_g(\widetilde{\mathbf{t}})=\sum_{b,k\geq 0} \frac{u^b}{k!} \langle\underbrace{\widetilde{\mathbf{t}}(\psi),...,\widetilde{\mathbf{t}}(\psi)}_{k} \rangle_{g,b}^\text{Sym}\, ,   
\end{split}    
\end{equation*}
and the total ancestor potentials\footnote{The results of \cite{pt} imply that $\mathcal{A}^\text{Hilb}_{an, t_D}$ depends analytically in $q$.},
\begin{equation*}
\begin{split}
&\mathcal{A}^\text{Hilb}_{an, t_D}=\exp\left(\sum_{g\geq 0}\hbar^{g-1}\cF_{an, t_D, g}^\text{Hilb} \right)\, , \quad \cF^\text{Hilb}_{an, t_D,g}(\mathbf{t})=\sum_{d,k, l\geq 0} \frac{q^d}{k!l!} \langle\underbrace{\mathbf{t}(\bar{\psi}),...,\mathbf{t}(\bar{\psi})}_{k}, \underbrace{t_D D,..., t_D D}_{l} \rangle_{g,d}^\text{Hilb}\, ,\\ 
&\mathcal{A}^\text{Sym}_{an, \tilde{t}}=\exp\left(\sum_{g\geq 0}\hbar^{g-1}\cF_{an, \tilde{t}, g}^\text{Sym} \right)\, , \quad \cF^\text{Sym}_{an, \tilde{t},g}(\widetilde{\mathbf{t}})=\sum_{b,k, l\geq 0} \frac{u^b}{k!l!} \langle\underbrace{\widetilde{\mathbf{t}}(\bar{\psi}),...,\widetilde{\mathbf{t}}(\bar{\psi})}_{k}, \underbrace{t I_{(2)},..., t I_{(2)}}_{l} \rangle_{g,b}^\text{Sym}\, .
\end{split}    
\end{equation*}

Givental's quantization formalism \cite{g} produces differential operators by quantizing quadratic Hamiltonians associated to linear symplectic transforms by the following rules:
\begin{equation*}
\begin{split}
&\widehat{q_a^\mu q_b^\nu}=\frac{q_a^\mu q_b^\nu}{\hbar}, \widehat{q_a^\mu p_b^\nu}=q_a^\mu \frac{\partial}{\partial q_b^\nu}, \widehat{p_a^\mu p_b^\nu}=\hbar\frac{\partial}{\partial q_a^\mu}\frac{\partial}{\partial q_b^\nu}\, ,\\
&\widehat{\widetilde{q}_a^\mu \widetilde{q}_b^\nu}=\frac{\widetilde{q}_a^\mu \widetilde{q}_b^\nu}{\hbar}, \widehat{\widetilde{q}_a^\mu \widetilde{p}_b^\nu}=\widetilde{q}_a^\mu \frac{\partial}{\partial \widetilde{q}_b^\nu}, \widehat{\widetilde{p}_a^\mu \widetilde{p}_b^\nu}=\hbar\frac{\partial}{\partial \widetilde{q}_a^\mu}\frac{\partial}{\partial \widetilde{q}_b^\nu}\, .
\end{split}    
\end{equation*}

By the descendent/ancestor relations \cite{cg}, we have 
\begin{equation*}
\begin{split}
&\mathcal{D}^\text{Hilb}=e^{F^\text{Hilb}_1(t_D)}\widehat{\mathsf{S}^\text{Hilb}(q, t_D)^{-1}}\mathcal{A}^\text{Hilb}_{an, t_D}\, ,\\
&\mathcal{D}^\text{Sym}=e^{F^\text{Sym}_1(\tilde{t})}\widehat{\mathsf{S}^\text{Sym}(u, \tilde{t})^{-1}}\mathcal{A}^\text{Sym}_{an, \tilde{t}}\, ,
\end{split}
\end{equation*}
where $F_1^\text{Hilb}$ and $F_1^\text{Sym}$ are generating functions of genus $1$ primary invariants with insertions $D$ and $I_{(2)}$ respectively. $F_1^\text{Sym}$ and $F_1^\text{Hilb}$ can be easily matched using \cite[Theorem 4]{pt}.


\begin{thm}\label{thm:higher_genus_corresp} We have 
$e^{-F^\text{Sym}_1(\tilde{t})}\mathcal{D}^\text{Sym}=\widehat{\mathsf{C}}\widehat{\mathsf{K}}\widehat{q^{-D/z}}\left(e^{-F^\text{Hilb}_1(t_D)}\mathcal{D}^\text{Hilb}\right)$.
\end{thm}

\begin{proof}
By \cite[Theorem 4]{pt}, we have 
$\widehat{\mathsf{C}}\mathcal{A}^\text{Hilb}_{an, t_D}=\mathcal{A}^\text{Sym}_{an, \tilde{t}}$ .  Using Theorem \ref{thm:S_matrices}, we calculate
\begin{equation*}
\widehat{\mathsf{S}^\text{Sym}(u, \tilde{t})^{-1}}\mathcal{A}^\text{Sym}_{an, \tilde{t}}\\
=\widehat{\mathsf{C}}\widehat{\mathsf{K}q^{-D/z}}\widehat{\mathsf{S}^\text{Hilb}(q, t_D)^{-1}}\mathcal{A}^\text{Hilb}_{an, t_D}\, .
\end{equation*}
Therefore, we conclude
\begin{equation*}
\begin{split}
e^{-F^\text{Sym}_1(\tilde{t})}\mathcal{D}^\text{Sym}&=\widehat{\mathsf{S}^\text{Sym}(u, \tilde{t})^{-1}}\mathcal{A}^\text{Sym}_{an, \tilde{t}}\\    
&=\widehat{\mathsf{C}}\widehat{\mathsf{K}q^{-D/z}}\widehat{\mathsf{S}^\text{Hilb}(q, t_D)^{-1}}\mathcal{A}^\text{Hilb}_{an, t_D}\\
&=\widehat{\mathsf{C}}\widehat{\mathsf{K}q^{-D/z}}\left(e^{-F^\text{Hilb}_1(t_D)}\mathcal{D}^\text{Hilb}\right)\, .
\end{split}
\end{equation*}
\end{proof}

\section{Fourier-Mukai transformation}
\subsection{Proof
of  Theorem \ref{thm:comparison}}
We first 
localize the top row 
of 
the diagram of Theorem
\ref{thm:comparison}:
\begin{displaymath}
    \xymatrix{ 
    K_\T(\hilbnc)_{\text{loc}} \ar[r]^{\mathbb{FM}}\ar[d]_{\Psi^\text{Hilb}} & K_\T(\text{Sym}^n(\mathbb{C}^2))_{\text{loc}}\ar[d]^{\Psi^\text{Sym}} \\
   {\widetilde{\mathcal{H}}^\text{Hilb}}\ar[r]^{\mathsf{C}\mathsf{K}\big|_{z\mapsto -z}}& {\widetilde{\mathcal{H}}^\text{Sym}}\, .}
\end{displaymath}
Here, $\text{loc}$ denotes tensoring by $\text{Frac}(R(\T))$, the field of fractions of the representation ring $R(\T)$ of the torus $\T$. The maps
$\Psi^{\text{Hilb}}$
and $\Psi^{\text{Sym}}$
are still well-defined since
the $\T$-equivariant
Chern character of
a representation is invertible.
The commutation of the above
diagram immediately
implies the commutation
of the diagram of Theorem
\ref{thm:comparison}.

Let $k_\lambda\in K_{\mathsf{T}}(\hilbnc)$ be the skyscraper sheaf supported on the fixed point indexed by $\lambda$. The set 
$\{ k_\lambda \big| 
\lambda\in \text{Part}(n)\}$ is a basis of $K_{\mathsf{T}}(\hilbnc)_{\text{loc}}$ as a $\text{Frac}(R(\T))$-vector space. 
The commutation of
the localized diagram is then a consequence
of the following equality: for all $\lambda \in \text{Part}(n)$,  
\begin{equation}\label{eqn:FM_vs_K}
\mathsf{CK}\big|_{z\mapsto -z}\circ \Psi^\text{Hilb}(k_\lambda)= \Psi^\text{Sym}\circ \mathbb{FM}(k_\lambda)\, .  
\end{equation}
To prove (\ref{eqn:FM_vs_K}), we will match the two sides
by explicit calculation.

\subsection{Iritani's Gamma class}
For a vector bundle $\mathcal{V}$ on a Deligne-Mumford stack $\mathcal{X}$,
$$\mathcal{V} \rightarrow \mathcal{X}\, ,$$
Iritani has defined a characteristic class called the {\em Gamma class}. Let $$I\mathcal{X}=\coprod_i \mathcal{X}_i$$ be the decomposition of the inertia stack $I\mathcal{X}$ into connected components.
By pulling back $\mathcal{V}$ to $I\mathcal{X}$ and restricting to $\mathcal{X}_i$, we obtain a vector bundle $\mathcal{V}\big|_{\mathcal{X}_i}$ on $\mathcal{X}_i$. The stabilizer element $g_i$ of $\mathcal{X}$ associated to the component $\mathcal{X}_i$ acts on $\mathcal{V}_{\mathcal{X}_i}$. The bundle
$\mathcal{V}\big|_{\mathcal{X}_i}$ 
decomposes under $g_i$ into a direct sum of eigenbundles $$\mathcal{V}\big|_{\mathcal{X}_i}=\oplus_{0\leq f<1} \mathcal{V}_{i,f}\, ,$$ where $g_i$ acts on $\mathcal{V}_{i,f}$ by multiplication by $\exp(2\pi\sqrt{-1}f)$. The orbifold Chern character of $\mathcal{V}$ is defined to be
\begin{equation}\label{eqn:orb_ch}
    \widetilde{\ch}(\mathcal{V})=\bigoplus_i \sum_{0\leq f<1}\exp(2\pi\sqrt{-1}f)\, \ch(\mathcal{V}_{i,f})\in H^*(I\mathcal{X})\,,
\end{equation}
where $\ch(-)$ is the usual Chern character. 

For each $i$ and $f$, let $\delta_{i, f, j}$, for $1\leq j\leq \text{rank}\, \mathcal{V}_{i,f}$, be the Chern roots of $\mathcal{V}_{i,f}$. Iritani's Gamma class\footnote{The substitution of cohomology classes into Gamma function makes sense because the Gamma function $\Gamma(1+x)$ has a power series expansion at $x=0$.} is defined to be 
\begin{equation}
\Gamma (\mathcal{V})=\bigoplus_i \prod_{0\leq f<1}\prod_{j=1}^{\text{rank}\, \mathcal{V}_{i,f}}\Gamma (1-f+\delta_{i,f,j})\, .    
\end{equation}
As usual, $\Gamma_\mathcal{X}=\Gamma(T\mathcal{X})$.

If the vector bundle $\mathcal{V}$ is equivariant with respect to a $\T$-action, the Chern character and Chern roots above should be replaced by their equivariant counterparts to define a
$\T$-equivariant Gamma class.

If $\mathcal{X}$ is a scheme, then
the Gamma class simplifies
considerably since there
are no stabilizers.
Directly from the definition,
the restriction of $\Gamma_\text{Hilb}$ to the fixed point indexed by $\lambda$ is 
\begin{equation*}
\Gamma_\text{Hilb}\Big|_{\lambda}=\prod_{\mathsf{w}: \text{ tangent weights at $\lambda$}}\Gamma(\mathsf{w}+1)\, .    
\end{equation*}

Recall that the inertia stack $I\text{Sym}^n(\mathbb{C}^2)$ is a disjoint union indexed by conjugacy classes of $S_n$. For a partition $\mu$ of $n$, the component $I_\mu\subset I\text{Sym}^n(\mathbb{C}^2)$ indexed by the conjugacy class of cycle type $\mu$ is the stack quotient $$[\mathbb{C}^{2n}_\sigma/C(\sigma)]\, ,$$ where $\sigma\in S_n$ has cycle type $\mu$, $\mathbb{C}^{2n}_\sigma\subset \mathbb{C}^{2n}$ is the $\sigma$-invariant part, and $C(\sigma)\subset S_n$ is the centralizer of $\sigma$.
\begin{lem}\label{lem:gamma_sym}
The restriction of $\Gamma_\text{\em Sym}$ to the component $I_\mu$ is given by 
\begin{equation*}
\Gamma_\text{\em Sym}\Big|_{\mu}=(t_1t_2)^{\ell(\mu)}(2\pi)^{n-\ell(\mu)}\left(\prod_i \mu_i\right) \left(\prod_i \mu_i^{-\mu_i t_1}\mu_i^{-\mu_i t_2}\right)\left(\prod_i \Gamma(\mu_i t_1)\Gamma(\mu_i t_2)\right).       
\end{equation*}
\end{lem}
\begin{proof}
Using the description of eigenspaces of $T_{\text{Sym}^n(\mathbb{C}^2)}$ on the component of $I\text{Sym}^n(\mathbb{C}^2)$ indexed by $\mu$ (see \cite[Section 6.2]{pt}), we find that $$\Gamma_\text{Sym}\Big|_{\mu}=\prod_i \prod_{l=0}^{\mu_i-1}\Gamma\left(1-\frac{l}{\mu_i}+t_1 \right) \Gamma\left(1-\frac{l}{\mu_i}+t_2 \right).$$
Using the formula
\begin{equation*}
\prod_{k=0}^{m-1}\Gamma\left(z+\frac{k}{m}\right)=(2\pi)^{\frac{m-1}{2}}m^{\frac{1}{2}-mz}\Gamma(mz)\, ,    
\end{equation*}
we find 
\begin{equation*}
\prod_{l=0}^{\mu_i-1}\Gamma\left(1-\frac{l}{\mu_i}+t_1 \right)=t_1(2\pi)^{\frac{\mu_i-1}{2}}\mu_i^{\frac{1}{2}-\mu_it_1}\Gamma(\mu_it_1)\, ,    
\end{equation*}
and similarly for the other factor. Therefore, 
\begin{equation*}
\begin{split}
\Gamma_\text{Sym}\Big|_\mu= (t_1t_2)^{\ell(\mu)}(2\pi)^{n-\ell(\mu)}\left(\prod_i \mu_i\right) \left(\prod_i \mu_i^{-\mu_i t_1}\mu_i^{-\mu_i t_2}\right)\left(\prod_i \Gamma(\mu_i t_1)\Gamma(\mu_i t_2)\right)\, ,
\end{split}    
\end{equation*}
which is the desired formula.
\end{proof}

\subsection{Calculation of $\mathsf{CK}\circ \Psi^\text{Hilb}$}
 Since $k_\lambda$ is supported at the $\T$-fixed point of $\hilbnc$ indexed by $\lambda$, the $\T$-equivariant Chern character $\ch(k_\lambda)$ is also
 supported there. Using the Koszul resolution (or Grothendieck-Riemann-Roch), we calculate 
\begin{equation}\label{eqn:chern_kp}
\ch(k_\lambda)= {\mathsf{J}^\lambda}\prod_{\mathsf{w}: \text{ tangent weights at $\lambda$}}
\frac{1-e^{-\mathsf{w}}}
{\mathsf{w}} 
\, .
\end{equation}
We have used the fact that the class of the $\T$-fixed point of $\hilbnc$ indexed by $\lambda$ corresponds to the factor $$\frac{\mathsf{J}^\lambda}{\prod_{\mathsf{w}}\mathsf{w}}\, .$$

By the definition of $\text{deg}_0^\text{Hilb}$, we have 
\begin{equation*}
(2\pi \sqrt{-1})^{\frac{\text{deg}^\text{Hilb}_0}{2}}\ch(k_\lambda)=\frac{(2\pi \sqrt{-1})^{\frac{\text{deg}^\text{Hilb}_0}{2}}\mathsf{J}^\lambda}{\prod_\mathsf{w} 2\pi\sqrt{-1}\mathsf{w}}\prod_{\mathsf{w}: \text{ tangent weights at $\lambda$}}(1-e^{-\mathsf{2\pi \sqrt{-1} w}})\, .    
\end{equation*}
Write $\mathsf{J}^\lambda=\sum_\epsilon \mathsf{J}^\lambda_\epsilon(t_1, t_2) |\epsilon\rangle$. Since $\mathsf{J}^\lambda_\epsilon$ is $(t_1t_2)^{\ell(\epsilon)}$ times a homogeneous polynomial in $t_1, t_2$ of degree $n-\ell(\epsilon)$, we have{\footnote{The calculation also follows from the fact that
$\mathsf{J}^\lambda$ is the class a $\T$-fixed point (of real degree $4n$).}} 
\begin{equation*}
\begin{split}
(2\pi \sqrt{-1})^{\frac{\text{deg}^\text{Hilb}_0}{2}}\mathsf{J}^\lambda=&\sum_\epsilon (2\pi \sqrt{-1})^{\frac{\text{deg}^\text{Hilb}_0}{2}}\mathsf{J}^\lambda_\epsilon(t_1, t_2)|\epsilon\rangle\\
=&\sum_\epsilon \mathsf{J}^\lambda_\epsilon(2\pi \sqrt{-1}t_1, 2\pi \sqrt{-1}t_2)(2\pi \sqrt{-1})^{n-\ell(\epsilon)}|\epsilon\rangle\\
=&\sum_\epsilon \mathsf{J}^\lambda_\epsilon(t_1, t_2)(2\pi \sqrt{-1})^{n+\ell(\epsilon)}(2\pi \sqrt{-1})^{n-\ell(\epsilon)}|\epsilon\rangle\\
=&(2\pi \sqrt{-1})^{2n}\sum_\epsilon \mathsf{J}^\lambda_\epsilon(t_1, t_2) |\epsilon\rangle\\
=&(2\pi \sqrt{-1})^{2n}\mathsf{J}^\lambda.
\end{split}    
\end{equation*}
After putting the above formulas together, we obtain
\begin{equation*}
\Gamma_\text{Hilb}\cup (2\pi \sqrt{-1})^{\frac{\text{deg}^\text{Hilb}_0}{2}}\ch(k_\lambda)=\frac{(2\pi \sqrt{-1})^{2n}\mathsf{J}^\lambda}{\prod_\mathsf{w}2\pi\sqrt{-1} \mathsf{w}}\prod_{\mathsf{w}: \text{ tangent weights at $\lambda$}}\Gamma(\mathsf{w}+1)(1-e^{-\mathsf{2\pi \sqrt{-1} w}})\, .
\end{equation*}

Recall the following identity for the Gamma function:
\begin{equation}\label{eqn:gamma_identity}
\Gamma(1+t)\Gamma(1-t)=\frac{2\pi \sqrt{-1} t}{e^{\pi\sqrt{-1}t}-e^{-\pi\sqrt{-1}t}}\, .    
\end{equation}
We have 
\begin{equation*}
\begin{split}
\Gamma(\mathsf{w}+1)(1-e^{\mathsf{-2\pi \sqrt{-1} w}})
=&\Gamma(\mathsf{w}+1)(e^{\pi\sqrt{-1}\mathsf{w}}-e^{-\pi \sqrt{-1} \mathsf{w}})(e^{-\pi \sqrt{-1} \mathsf{w}})\\
=&\frac{2\pi\sqrt{-1}\mathsf{w}}{\Gamma(1-\mathsf{w})}(e^{-\pi \sqrt{-1} \mathsf{w}})\, .
\end{split}
\end{equation*}
Hence 
\begin{equation*}
\Gamma_\text{Hilb}\cup (2\pi \sqrt{-1})^{\frac{\text{deg}^\text{Hilb}_0}{2}}\ch(k_\lambda)=((2\pi \sqrt{-1})^{2n}\mathsf{J}^\lambda)\prod_{\mathsf{w}: \text{ tangent weights at $\lambda$}}\frac{1}{\Gamma(1-\mathsf{w})}e^{-\pi \sqrt{-1} \mathsf{w}}\, .    
\end{equation*}

Since the operator $z^{\rho^\text{Hilb}}$ is the operator of multiplication by $z^{c_1^\T(\hilbnc)}$, we have 
\begin{multline*}
z^{\rho^\text{Hilb}}\left(\Gamma_\text{Hilb}\cup (2\pi \sqrt{-1})^{\frac{\text{deg}^\text{Hilb}_0}{2}}\ch(k_\lambda)\right)\\
=z^{{n(t_1+t_2)}}((2\pi \sqrt{-1})^{2n}\mathsf{J}^\lambda)\prod_{\mathsf{w}: \text{ tangent weights at $\lambda$}}\frac{1}{\Gamma(1-\mathsf{w})}e^{-\pi \sqrt{-1} \mathsf{w}}\\
=z^{{n(t_1+t_2)}}e^{-\pi\sqrt{-1}n(t_1+t_2)}((2\pi \sqrt{-1})^{2n}\mathsf{J}^\lambda)\prod_{\mathsf{w}: \text{ tangent weights at $\lambda$}}\frac{1}{\Gamma(1-\mathsf{w})}\, ,
\end{multline*}
where we use $$c_1^\T(\hilbnc)\Big|_\lambda=\sum_{\mathsf{w}: \text{ tangent weights at $\lambda$}}\mathsf{w}=n(t_1+t_2)\, .$$

By the definition of $\mu^\text{Hilb}$, we have  $$z^{-\mu^{\text{Hilb}}}(\phi)=z^n z^{-\text{deg}_0^\text{Hilb}/2}(\phi)=z^n(\frac{\phi}{z^{k/2}})$$ for $\phi\in H_\T^k(\hilbnc, \mathbb{C})$, we have 
\begin{multline*}
z^{-\mu^\text{Hilb}}z^{\rho^\text{Hilb}}\left(\Gamma_\text{Hilb}\cup (2\pi \sqrt{-1})^{\frac{\text{deg}^\text{Hilb}_0}{2}}\ch(k_\lambda)\right)\\
=z^nz^{n(t_1+t_2)/z}e^{-\pi\sqrt{-1}n(t_1+t_2)/z}\left(\frac{2\pi \sqrt{-1}}{z}\right)^{2n}\mathsf{J}^\lambda \prod_{\mathsf{w}: \text{ tangent weights at $\lambda$}}\frac{1}{\Gamma(1-\mathsf{w}/z)}\, .
\end{multline*}
Here, the operator $z^{-\text{deg}_0^\text{Hilb}/2}$ acts on $z^{n(t_1+t_2)}$ as follows:
\begin{equation*}
\begin{split}
 z^{-\text{deg}_0^\text{Hilb}/2}(z^{n(t_1+t_2)})=& z^{-\text{deg}_0^\text{Hilb}/2}(e^{n(t_1+t_2)\log z})\\
 =&z^{-\text{deg}_0^\text{Hilb}/2}\left(\sum_{k\geq 0}\frac{(n(t_1+t_2)\log z)^k}{k!}\right)\\
 =&\sum_{k\geq 0} \frac{(n\log z)^k z^{-\text{deg}_0^\text{Hilb}/2}((t_1+t_2)^k)}{k!}\\
 =&\sum_{k\geq 0} \frac{(n\log z)^k ((t_1+t_2)^k/z^k)}{k!}\\
 =&\sum_{k\geq 0} \frac{(n\log z ((t_1+t_2)/z))^k}{k!}\\
 =&z^{n(t_1+t_2)/z}.
\end{split}    
\end{equation*}
The actions of $z^{-\text{deg}_0^\text{Hilb}/2}$ on $e^{-\pi\sqrt{-1}n(t_1+t_2)}$ and $\Gamma(1+\mathsf{w})$ are similarly determined.

By Equation (\ref{prop:op_K}), we have 
\begin{equation*}
\mathsf{K}\big|_{z\mapsto -z}(\mathsf{J}^\lambda)= \frac{(-z)^{|\lambda|}}{(2\pi\sqrt{-1})^{|\lambda|}}\left(\prod_{\mathsf{w}: \text{ tangent weights at $\lambda$}}\Gamma(-\mathsf{w}/z+1)\right)\Theta' \mathbf{\Gamma}_{-z}\mathsf{H}^\lambda_{-z}\, ,  \end{equation*}
where we define $\Theta' |\mu\rangle= (-z)^{\ell(\mu)}|\mu\rangle$ .
Hence,
\begin{equation*}
\begin{split}
&\mathsf{K}\big|_{z\mapsto -z}\left(z^{-\mu^\text{Hilb}}z^{\rho^\text{Hilb}}\left(\Gamma_\text{Hilb}\cup (2\pi \sqrt{-1})^{\frac{\text{deg}^\text{Hilb}_0}{2}}\ch(k_\lambda)\right)\right)\\
=&z^n z^{n(t_1+t_2)/z}e^{-\pi\sqrt{-1}n(t_1+t_2)/z}\left(\frac{2\pi \sqrt{-1}}{z}\right)^{2n}\mathsf{K}\big|_{z\mapsto -z}(\mathsf{J}^\lambda) \prod_{\mathsf{w}: \text{ tangent weights at $\lambda$}}\frac{1}{\Gamma(1-\mathsf{w}/z)}\\
=&z^n z^{n(t_1+t_2)/z}e^{-\pi\sqrt{-1}n(t_1+t_2)/z}\left(\frac{2\pi \sqrt{-1}}{z}\right)^{2n}\frac{(-z)^{|\lambda|}}{(2\pi\sqrt{-1})^{|\lambda|}}\Theta' \mathbf{\Gamma}_{-z}\mathsf{H}^\lambda_{-z} \prod_{\mathsf{w}: \text{ tangent weights at $\lambda$}}\frac{\Gamma(-\mathsf{w}/z+1)}{\Gamma(1-\mathsf{w}/z)}\\
=&(-1)^n z^n z^{n(t_1+t_2)/z}e^{-\pi\sqrt{-1}n(t_1+t_2)/z}\left(\frac{2\pi \sqrt{-1}}{z}\right)^{n}\Theta' \mathbf{\Gamma}_{-z}\mathsf{H}^\lambda_{-z}\, . 
\end{split}    
\end{equation*}

By the definition
of $\mathbf{\Gamma}_{-z}$, we have 
\begin{equation*}
\mathbf{\Gamma}_{-z}|\mu\rangle=\frac{(2\pi\sqrt{-1})^{\ell(\mu)}}{\prod_i \mu_i}\prod_i \frac{\mu_i^{-\mu_it_1/z}\mu_i^{-\mu_i t_2/z}}{\Gamma(-\mu_i t_1/z)\Gamma(-\mu_i t_2/z)}|\mu\rangle\, .
\end{equation*}
Also, $\mathsf{C}|\mu\rangle=|\widetilde{\mu}\rangle$, we thus obtain 
\begin{equation}\label{eqn:hilb_side}
\begin{split}
&\mathsf{CK}\big|_{z\mapsto -z}\left(z^{-\mu^\text{Hilb}}z^{\rho^\text{Hilb}}\left(\Gamma_\text{Hilb}\cup (2\pi \sqrt{-1})^{\frac{\text{deg}^\text{Hilb}_0}{2}}\ch(k_\lambda)\right)\right)
=\Delta^{\text{Hilb}}(\mathsf{H}^\lambda_{-z})\, ,
\end{split}    
\end{equation}
where $\Delta^\text{Hilb}: \cF\to \widetilde{\cF}$ is the operator defined as follows:
{\footnotesize{
\begin{equation}\label{eqn:delta_H}
\begin{split}
&\Delta^\text{Hilb}|\mu\rangle\\
=&(-1)^n z^n z^{n(t_1+t_2)/z}e^{-\pi\sqrt{-1}n(t_1+t_2)/z}\left(\frac{2\pi \sqrt{-1}}{z}\right)^{n}(-z)^{\ell(\mu)}\frac{(2\pi\sqrt{-1})^{\ell(\mu)}}{\prod_i \mu_i}\prod_i \frac{\mu_i^{-\mu_it_1/z}\mu_i^{-\mu_i t_2/z}}{\Gamma(-\mu_i t_1/z)\Gamma(-\mu_i t_2/z)}|\widetilde{\mu}\rangle\\
=&(-1)^{n+\ell(\mu)}z^{n(t_1+t_2)/z}e^{-\pi\sqrt{-1}n(t_1+t_2)/z}(2\pi\sqrt{-1})^{n+\ell(\mu)}z^{\ell(\mu)}\frac{1}{\prod_i \mu_i}\prod_i \frac{\mu_i^{-\mu_it_1/z}\mu_i^{-\mu_i t_2/z}}{\Gamma(-\mu_i t_1/z)\Gamma(-\mu_i t_2/z)}|\widetilde{\mu}\rangle\, . 
\end{split}
\end{equation} }}

\subsection{Haiman's result}
The homomorphism $\mathbb{FM}$ has been calculated by Haiman \cite{haiman_cdm,haiman_sym2001}. Denote by $F$ the operator of taking Frobenius series of bigraded $S_n$-modules, as defined in \cite[Definition 3.2.3]{haiman_cdm}. Note that $\T$-equivariant sheaves on $$\text{Sym}^n(\mathbb{C}^2)=[(\mathbb{C}^2)^{n}/S_n]$$ are $\T\times S_n$-equivariant sheaves on $\mathbb{C}^2$, and hence can be identified with bigraded $S_n$-equivariant $\mathbb{C}[\mathbf{x}, \mathbf{y}]$-modules\footnote{Here, $\mathbf{x}=\{x_1,...,x_n\}$ and $\mathbf{y}=\{y_1,...,y_n\}$.}. Therefore, the composition $$\Phi=F\circ \mathbb{FM}$$ makes sense and takes values in a certain algebra of symmetric functions, see \cite[Proposition 5.4.6]{haiman_cdm}. For the analysis of the
diagram of Theorem \ref{thm:comparison}, we will need the following result of Haiman.

\begin{thm}[\cite{haiman_cdm}, Equation (95)]
Let $k_\lambda\in K_\T(\hilbnc)$ be the skyscraper sheaf supported on the $\T$-fixed point indexed by $\lambda$. Then 
$$\Phi(k_\lambda)=\widetilde{H}_\lambda(z; q, t)\, .$$
\end{thm}

The Macdonald polynomial $\widetilde{H}_\lambda(z; q, t)$ is a symmetric function in an infinite set of variables $$z=\{z_1,z_2,z_3, ...\}$$ 
and depends on two parameters $q, t$.
As explained in \cite[Section 9.1]{pt}, $\widetilde{H}_\lambda(z;q,t)$ of \cite{haiman_cdm} is the same as $\mathsf{H}^\lambda$ after the
following identification: the parameters $(q, t)$ and $(t_1, t_2)$ are related by $$(q, t)=(e^{2\pi\sqrt{-1}t_1}, e^{2\pi\sqrt{-1}t_2})\, .$$ Symmetric functions in $z$ are viewed as elements of $\widetilde{\mathcal{F}}$ via the following convention. For a partition $\mu$, the power-sum symmetric function $$p_\mu=\prod_k\big(\sum_{i\geq 1}z_i^{\mu_k}\big)$$ is identified with $\mathfrak{z}(\mu)|\mu\rangle$.

To make use of Haiman's result, we must 
compare the operator $F$  taking Frobenius series with the orbifold Chern character $\widetilde{\ch}$. Let $V^\lambda$ be the irreducible $S_n$-representation indexed by $\lambda\in \text{Part}(n)$. We construct the bigraded $S_n$-equivariant $\mathbb{C}[\mathbf{x}, \mathbf{y}]$-module $V^\lambda\otimes\mathbb{C}[\mathbf{x},\mathbf{y}]$, which is equivalent to a $\T$-equivariant sheaf $\mathcal{V}^\lambda$ on $\text{Sym}^n(\mathbb{C}^2)$. Define the operator $\delta:\widetilde{\mathcal{F}}\to \widetilde{\mathcal{F}}$ by $$\delta |\mu\rangle=\prod_i (1-q^{\mu_i})(1-t^{\mu_i})|\mu\rangle\, .$$
By \cite[Section 5.4.3]{haiman_cdm}, we have $$F_{V^\lambda\otimes\mathbb{C}[\mathbf{x},\mathbf{y}]}=s_\lambda\Big[\frac{Z}{(1-q)(1-t)}\Big],$$
where $s_\lambda$ is the Schur function. Here $Z$ denotes the collection of variables $z_1,z_2,...$ that the functions are symmetric with respect to, according to the convention of \cite{haiman_cdm}. Using the definition of plethystic substitution $Z\mapsto Z/(1-q)(1-t)$, see  \cite[Section 3.3]{haiman_cdm}, we obtain $$\delta (F_{V^\lambda\otimes\mathbb{C}[\mathbf{x},\mathbf{y}]})=s_\lambda.$$ 
On the other hand, by the definition of orbifold Chern character\footnote{The natural basis of $H_\T^*(I\text{Sym}^n(\mathbb{C}^2))$ is identified with $\{|{\mu}\rangle \big| \mu\in \text{Part}(n)\}\subset \widetilde{\cF}$.} recalled in Equation \eqref{eqn:orb_ch}, we have $$\widetilde{\ch}(\mathcal{V}^\lambda)=s_\lambda\, .$$
Since $K_\T(\text{Sym}^n(\mathbb{C}^2))$ is freely spanned as a $R(T)$-module by $V^\lambda\otimes\mathbb{C}[\mathbf{x},\mathbf{y}]$, we find $$\delta\circ F=\widetilde{\ch}\, ,$$ 
after identifying\footnote{The choice of $\T=(\mathbb{C}^*)^2$-action on $\mathbb{C}^2$ in \cite[Section 5.1.1]{haiman_cdm} is dual to ours.} $q=e^{-t_1}, t=e^{-t_2}$. Therefore, 
\begin{equation*}
\begin{split}
\widetilde{\ch}(\mathbb{FM}(k_\lambda))
=&\delta(F(\mathbb{FM}(k_\lambda)))\\
=&\delta(\Phi(k_\lambda))\\
=&\delta(\tilde{H}_\lambda)\, , \quad q=e^{-t_1}\,
, \ \ t=e^{-t_2}.
\end{split}
\end{equation*}

\subsection{Calculation of $\Psi^\text{Sym}\circ \mathbb{FM}$}
We have 
\begin{equation*}
(2\pi\sqrt{-1})^{\frac{\text{deg}^\text{Sym}_0}{2}}\widetilde{\ch}(\mathbb{FM}(k_\lambda))=\delta(\widetilde{H}_\lambda)\, , \quad q=e^{-2\pi\sqrt{-1}t_1}\, ,\ \  t=e^{-2\pi\sqrt{-1}t_2}\, .
\end{equation*}
We have used the definition of $\text{deg}^\text{Sym}_0$ and the fact that $|\mu\rangle\in\widetilde{\cF}$ as a class in $H_\T^*(I\text{Sym}^n(\mathbb{C}^2))$ has degree $0$.

By Lemma \ref{lem:gamma_sym}, we have 
\begin{equation*}
\Gamma_\text{Sym}\cup (2\pi\sqrt{-1})^{\frac{\text{deg}^\text{Sym}_0}{2}}\widetilde{\ch}(\mathbb{FM}(k_\lambda))=\delta_2(\widetilde{H}_\lambda)\, , \quad q=e^{-2\pi\sqrt{-1}t_1}\, ,\ \ t=e^{-2\pi\sqrt{-1}t_2},   
\end{equation*}
where $\delta_2: \widetilde{\mathcal{F}}\to \widetilde{\mathcal{F}}$ is defined by 
\begin{multline*}
\delta_2|\mu\rangle
=(t_1t_2)^{\ell(\mu)}(2\pi)^{n-\ell(\mu)}\left(\prod_i \mu_i\right) \left(\prod_i \mu_i^{-\mu_i t_1}\mu_i^{-\mu_i t_2}\right)\\
\times \left(\prod_i \Gamma(\mu_i t_1)\Gamma(\mu_i t_2)\right)\left(\prod_i (1-e^{-2\pi\sqrt{-1}\mu_i t_1})(1-e^{-2\pi\sqrt{-1}\mu_i t_2}) \right)|{\mu}\rangle\, .    
\end{multline*}
Since $c_1^\T(\text{Sym}^n(\mathbb{C}^2))\Big|_\mu=n(t_1+t_2)$,
we have 
\begin{equation*}
z^{\rho^\text{Sym}}\left(\Gamma_\text{Sym}\cup (2\pi\sqrt{-1})^{\frac{\text{deg}^\text{Sym}_0}{2}}\widetilde{\ch}(\mathbb{FM}(k_\lambda)) \right)=z^{n(t_1+t_2)}\delta_2(\widetilde{H}_\lambda)\, , \quad q=e^{-2\pi\sqrt{-1}t_1}\, ,\ \ t=e^{-2\pi\sqrt{-1}t_2}.
\end{equation*}

Next, we write 
\begin{equation*}
z^{-\mu^\text{Sym}}z^{\rho^\text{Sym}}\left(\Gamma_\text{Sym}\cup (2\pi\sqrt{-1})^{\frac{\text{deg}^\text{Sym}_0}{2}}\widetilde{\ch}(\mathbb{FM}(k_\lambda)) \right)=\delta_3(\mathsf{H}_{-z}^\lambda)\, ,    
\end{equation*}
where $\delta_3: \widetilde{\mathcal{F}}\to \widetilde{\mathcal{F}}$ is defined by 
\begin{multline*}
\delta_3|\mu\rangle
=z^n z^{n(t_1+t_2)/z}(t_1t_2/z^2)^{\ell(\mu)}(2\pi)^{n-\ell(\mu)}\left(\prod_i \mu_i\right)\left(\prod_i \mu_i^{-\mu_i t_1/z}\mu_i^{-\mu_i t_2/z}\right)\\
\times \left(\prod_i \Gamma(\mu_i t_1/z)\Gamma(\mu_i t_2/z)\right)\left(\prod_i (1-e^{-2\pi\sqrt{-1}\mu_it_1/z})(1-e^{-2\pi\sqrt{-1}\mu_it_2/z}) \right)z^{-(n-\ell(\mu))}|{\mu}\rangle\, .
\end{multline*}
We have used the definition of $\mu^\text{Sym}$ and the fact that $|\mu\rangle\in\widetilde{\cF}$ as a class in $H_\T^*(I\text{Sym}^n(\mathbb{C}^2))$ has age-shifted degree $2(n-\ell(\mu))$. We have
also used
\begin{equation*}
\begin{split}
z^{\text{deg}_\text{CR}/2}\big(\widetilde{H}_\lambda\big|_{ q=e^{-2\pi\sqrt{-1}t_1},\ t=e^{-2\pi\sqrt{-1}t_2}}\big)= \widetilde{H}_\lambda\big|_{q=e^{-2\pi\sqrt{-1}t_1/z},\ t=e^{-2\pi\sqrt{-1}t_2/z}}\, ,   
\end{split}    
\end{equation*}  
which is equal to $\mathsf{H}^\lambda_{-z}$.

By (\ref{eqn:gamma_identity}), we have 
\begin{equation*}
\begin{split}
\Gamma(t)\Gamma(-t)=&\frac{\Gamma(1+t)}{t}\frac{\Gamma(1-t)}{-t}\\
=&\frac{1}{-t}\frac{2\pi \sqrt{-1}}{e^{\pi\sqrt{-1}t}-e^{-\pi\sqrt{-1}t}}\\
=&\frac{2\pi\sqrt{-1}}{-t}\frac{1}{(1-e^{-2\pi\sqrt{-1}t})e^{\pi\sqrt{-1}t}}\, .  
\end{split}    
\end{equation*}
Hence 
\begin{equation*}
\Gamma(t)(1-e^{-2\pi\sqrt{-1}t})=(-1)e^{-\pi\sqrt{-1}t}2\pi\sqrt{-1} \frac{1}{t}\frac{1}{\Gamma(-t)}\, .    
\end{equation*}
We then obtain
\begin{equation*}
\begin{split}
 &\left(\prod_i \Gamma(\mu_i t_1/z)\Gamma(\mu_i t_2/z)\right)\left(\prod_i (1-e^{-2\pi\sqrt{-1}\mu_i t_1/z})(1-e^{-2\pi\sqrt{-1}\mu_i t_2/z}) \right)\\
 =&(-1)^{2\ell(\mu)}e^{-\pi\sqrt{-1}n(t_1+t_2)/z}(2\pi\sqrt{-1})^{2\ell(\mu)}\left(\prod_i \frac{z}{\mu_i t_1}\frac{z}{\mu_i t_2}\right)\left(\prod_i\frac{1}{\Gamma(-\mu_i t_1/z)\Gamma(-\mu_i t_2/z)} \right)\\
 =&(-1)^{2\ell(\mu)}e^{-\pi\sqrt{-1}n(t_1+t_2)/z}(2\pi\sqrt{-1})^{2\ell(\mu)}\left(\frac{z^2}{t_1t_2}\right)^{\ell(\mu)}\left(\prod_i \frac{1}{\mu_i}\right)^2\left(\prod_i\frac{1}{\Gamma(-\mu_i t_1/z)\Gamma(-\mu_i t_2/z)} \right).
\end{split}    
\end{equation*}
Therefore, we can write $\delta_3|\mu\rangle$ as
\begin{equation*}
\begin{split}
&z^n z^{n(t_1+t_2)/z}(t_1t_2/z^2)^{\ell(\mu)}(2\pi)^{n-\ell(\mu)}\left(\prod_i \mu_i\right)\left(\prod_i \mu_i^{-\mu_i t_1/z}\mu_i^{-\mu_i t_2/z}\right)\\
&\times (-1)^{2\ell(\mu)}e^{-\pi\sqrt{-1}n(t_1+t_2)/z}(2\pi\sqrt{-1})^{2\ell(\mu)}\left(\frac{z^2}{t_1t_2}\right)^{\ell(\mu)}\left(\prod_i \frac{1}{\mu_i}\right)^2\\
&\times \left(\prod_i\frac{1}{\Gamma(-\mu_i t_1/z)\Gamma(-\mu_i t_2/z)} \right)z^{-(n-\ell(\mu))} |\mu\rangle\\
=\ &z^{\ell(\mu)}z^{n(t_1+t_2)/z}e^{-\pi\sqrt{-1}n(t_1+t_2)/z} \frac{1}{\prod_i \mu_i}\prod_i \frac{\mu_i^{-\mu_i t_1/z}\mu_i^{-\mu_i t_2/z}}{\Gamma(-\mu_i t_1/z)\Gamma(-\mu_i t_2/z)}\\
&\times (2\pi)^{n-\ell(\mu)}(2\pi\sqrt{-1})^{2\ell(\mu)}(-1)^{2\ell(\mu)}|\mu\rangle\, .
\end{split}    
\end{equation*}

\subsection{Proof of Theorem \ref{thm:comparison}}
The last step of the proof is the matching
\begin{equation}\label{eqn:final_eqn_class}
 \delta_3|\mu\rangle=\Delta^{\text{Hilb}}|\mu\rangle\, .   
\end{equation}
By comparing the expression above for $\delta_3|\mu\rangle$ with Equation (\ref{eqn:delta_H}), we see the matching \eqref{eqn:final_eqn_class} follows from the following equality in $\widetilde{\cF}$:
\begin{equation}\label{eqn:final_eqn}
(-1)^{n+\ell(\mu)}(2\pi\sqrt{-1})^{n+\ell(\mu)}|\widetilde{\mu}\rangle= (2\pi)^{n-\ell(\mu)}(2\pi\sqrt{-1})^{2\ell(\mu)}(-1)^{2\ell(\mu)}|\mu\rangle\, .   
\end{equation}
We verify \eqref{eqn:final_eqn} as follows. By definition, $|\widetilde{\mu}\rangle=(-\sqrt{-1})^{\ell(\mu)-n}|\mu\rangle$. Thus, 
\begin{equation*}
(-1)^{n+\ell(\mu)}(2\pi\sqrt{-1})^{n+\ell(\mu)}|\widetilde{\mu}\rangle=(-1)^{n+\ell(\mu)}(2\pi\sqrt{-1})^{n+\ell(\mu)}(-\sqrt{-1})^{\ell(\mu)-n}|\mu\rangle\, .    
\end{equation*}
We calculate
\begin{equation*}
\begin{split}
&(-1)^{n+\ell(\mu)}(2\pi\sqrt{-1})^{n+\ell(\mu)}(-\sqrt{-1})^{\ell(\mu)-n}=(2\pi)^{n+\ell(\mu)}(-1)^{2\ell(\mu)}\sqrt{-1}^{2\ell(\mu)}\, ,\\
&(2\pi)^{n-\ell(\mu)}(2\pi\sqrt{-1})^{2\ell(\mu)}(-1)^{2\ell(\mu)}=(2\pi)^{n+\ell(\mu)}(-1)^{2\ell(\mu)}\sqrt{-1}^{2\ell(\mu)}.
\end{split}    
\end{equation*}
This proves \eqref{eqn:final_eqn}, hence \eqref{eqn:final_eqn_class}.

In summary, our calculations establish 
the equation 
\begin{equation*}
\begin{split}
&z^{-\mu^\text{Sym}}z^{\rho^\text{Sym}}\left(\Gamma_\text{Sym}\cup (2\pi\sqrt{-1})^{\frac{\text{deg}^\text{Sym}_0}{2}}\widetilde{\ch}(\mathbb{FM}(k_\lambda)) \right)\\
=\  &\mathsf{CK}\big|_{z\mapsto -z}\left(z^{-\mu^\text{Hilb}}z^{\rho^\text{Hilb}}\left(\Gamma_\text{Hilb}\cup (2\pi \sqrt{-1})^{\frac{\text{deg}^\text{Hilb}_0}{2}}\ch(k_\lambda)\right)\right)\, ,
\end{split}
\end{equation*}
which completes the proof of Theorem \ref{thm:comparison} .
\qed

\end{document}